\documentclass[11pt,reqno]{amsart}
\usepackage[T1]{fontenc}
\usepackage{graphicx,   amssymb, mathtools}
\usepackage{cite}

\usepackage{color}
\definecolor{MyLinkColor}{rgb}{0,0,0.4}

\newcommand{\R}{{\mathbb R}}
\newcommand{\C}{{\mathbb C}}

\newcommand{\bA}{{\mathbb A}}
\newcommand{\bB}{\mathbb{B}}

\newcommand{\N}{{\mathbb N}}
\newcommand{\A}{\mathcal{A}}
\newcommand{\B}{\mathcal{B}}
\newcommand{\cO}{\mathcal{O}}
\newcommand{\cU}{\mathcal{U}}
\newcommand{\cV}{\mathcal{V}}
\newcommand{\cF}{\mathcal{F}}
\newcommand{\kL}{\mathcal{L}}
\newcommand{\re}{\mathop{\rm Re}\nolimits}

\newcommand{\PV}{\mathop{\rm PV}\nolimits}

\newcommand{\supp}{\mathop{\rm supp}\nolimits}

\newcommand{\wt}{\widetilde}
\newcommand{\ov}{\overline}
\newcommand{\oo}{\overline\omega}
\newcommand{\p}{\partial}
\newcommand{\e}{\varepsilon}
\newcommand{\0}{\Omega}
\newcommand{\G}{\Gamma}
\newtheorem{thm}{Theorem}[section]

\newtheorem{lemma}[thm]{Lemma}
\newtheorem{cor}[thm]{Corollary}
\theoremstyle{remark} 
\newtheorem{rem}[thm]{Remark}

\numberwithin{equation}{section} 

\title[The multiphase Muskat problem   in 2D]{The multiphase Muskat problem with general viscosities in two dimensions}
\thanks{Partially supported by DFG Research Training Group~2339 ``Interfaces, Complex Structures, and Singular Limits in Continuum Mechanics - Analysis and Numerics''}
\author{Jonas Bierler}
\address{Fakult\"at f\"ur Mathematik, Universit\"at Regensburg,   93040 Regensburg, Deutschland.}
\email{jonas.bierler@ur.de}
 \email{bogdan.matioc@ur.de}

\author{Bogdan--Vasile Matioc}


\usepackage[colorlinks=true,linkcolor=MyLinkColor,citecolor=MyLinkColor]{hyperref} 

\begin{document}

\begin{abstract}
In this paper we study the two-dimensional multiphase Muskat problem describing the motion of three immiscible fluids with general viscosities in a vertical homogeneous porous medium under the influence of gravity.
Employing  Rellich type identities in the regime where the fluids are ordered according to their viscosities, respectively a Neumann series argument  when the fluids are  not ordered by viscosity,  
we may   recast the  governing equations  as a strongly coupled nonlinear and nonlocal evolution problem for the functions that parameterize the sharp interfaces that separate the fluids.  
This problem is of parabolic type if the Rayleigh-Taylor condition is satisfied at each interface.
Based on this property, we then show that the multiphase Muskat problem is well-posed in all $L_2$-subcritical Sobolev spaces and that
it features  some parabolic smoothing  properties.
\end{abstract}
  
\maketitle

\section{Introduction}\label{Sec:1}

\subsection{The multiphase Muskat problem}
Three-phase flows in porous media are nowadays of great interest in many application such as seawater or CO$_2$ injection for enhanced oil recovery or the  CO$_2$ storage 
in depleted oilfields, cf. e.g. \cite{ASL20, BT12, AP14}. 
In this paper we study the multidimensional Muskat problem which is a classical model for  the motion of three incompressible fluids with positive constant densities 
\begin{equation}\label{stre}
\rho_3>\rho_2>\rho_1
\end{equation}
in a vertical porous medium. 
In our setting the porous medium is homogeneous with permeability constant $k>0$, the flow is two-dimensional, and  the fluid phase cover the entire plane~$\R^2$.
Moreover, the fluids are assumed to be separated  at each time instant~${t\geq0}$ by sharp interfaces which are parameterized as  graphs over the real line, that is
 \[
\Gamma_{f}^{c_\infty}(t)\coloneqq{}\{(x,c_\infty+f(t,x))\,:\, x\in\R\} \qquad\text{and}\qquad  \Gamma_h(t)\coloneqq{}\{(x,h(t,x))\,:\, x\in\R\},
 \] 
where $c_\infty$ is a fixed positive constant.  
We restrict to the nondegenerate situation where the   distance between the graphs $\Gamma_{f}^{c_\infty}(t)$ and $\Gamma_h(t)$ is positive during the flow.
Moreover, the  fluids layers are arranged according to their density.
 More precisely,  the fluid with  density $\rho_i$ is located at $\0_i(t)\subset\R^2$, $1\leq i\leq 3$, where $\0_2(t)\coloneqq{}\R^2\setminus\overline{\0_1(t)\cup\0_3(t)},$ 
\begin{align*}
& \0_1(t)\coloneqq{}\{(x,y)\in\R^2\,:\, y>f(t,x)+c_\infty\},\\[1ex]
& \0_3(t)\coloneqq{}\{(x,y)\in\R^2\,:\, y<h(t,x)\}.
\end{align*}

 Since the fluids are incompressible and the flow occurs at low Reynolds numbers,   the motion in the fluid layers    is governed by the following equations 
 \begin{subequations}\label{PB}
\begin{equation}\label{eq:S1}
\left.\begin{array}{rllllll}
v_i(t)\!\!\!\!&=&\!\!\!\!-\cfrac{k}{\mu_i}\big(\nabla p_i(t)+(0,\rho_i g)\big)   , \\[1ex]
{\rm div}\,  v_i(t)\!\!\!\!&=&\!\!\!\!0
\end{array}
\right\}\qquad\text{in $ \0_i(t)$, $1\leq i\leq 3$},
\end{equation} 
  where $\mu_i,$ $p_i(t),$  and $v_i(t)\coloneqq{}(v_i^1(t),v_i^2(t))$ is the viscosity, pressure, and  velocity, respectively, of the fluid located at $\0_i(t)$.
  The positive constant $g$ is the Earth's gravity. 
  The equation $\eqref{eq:S1}_1$ is Darcy's law which is the standard  model  for flows in porous media, cf. e.g. \cite{Be88}.
 
Neglecting surface tension effects,  the equations \eqref{eq:S1} are supplemented by the boundary conditions 
\begin{equation}\label{eq:S2}
\left.\begin{array}{rllllll}
p_i(t)\!\!\!\!&=&\!\!\!\!p_{i+1}(t), \\[1ex]
 \langle v_i(t)| \nu_i(t)\rangle\!\!\!\!&=&\!\!\!\!  \langle v_{i+1}(t)| \nu_i(t)\rangle 
\end{array}
\right\}\qquad\text{on $\p\0_i(t)\cap\p\0_{i+1}(t)$, $i=1,\, 2$,}
\end{equation} 
where    $\nu_i(t) $ is the unit normal at $\p\0_i(t)\cap\p\0_{i+1}(t)$ pointing into $\0_{i}(t)$ and~${\langle \, \cdot\,|\,\cdot\,\rangle}$  the Euclidean scalar product in~$\R^2$.
Additionally, the flow should satisfy  the  far-field boundary  conditions
 \begin{equation}\label{eq:S3}
\left.\begin{array}{rlllll}
v_i(t,x,y)\!\!\!\!&\to 0 \qquad\text{for  $|(x,y)|\to\infty$,  $1\leq i\leq 3$},\\[1ex]
f^2(t,x)+h^2(t,x)\!\!\!\!&\to 0 \qquad\text{for  $|x|\to\infty$,}
\end{array}\right\}
\end{equation}
that is far away the flow  is nearly stationary.

Finally,  the normal velocity of the interfaces  equals  the normal component of the velocity field at the free boundary, meaning that
 \begin{equation}\label{eq:S4}
  \left.\begin{array}{rllllll}
 \p_tf(t)\!\!\!\!&=&\!\!\!\! \langle v_1(t)| (-\p_x f(t),1)\rangle \qquad\text{on   $\Gamma_{f}^{c_\infty}(t),$}\\[1ex]
  \p_th(t)\!\!\!\!&=&\!\!\!\! \langle v_2(t)| (-\p_x h(t),1)\rangle \qquad\text{on   $\G_h(t),$}
 \end{array}\right\}
\end{equation}  
and  the  interfaces are assumed to be known initially
\begin{equation}\label{eq:S5}
(f,h)(0,\cdot )= (f_0,h_0).
\end{equation} 
\end{subequations}

\subsection{Summary of results.} Despite its physical relevance, the well-posedness of the multiphase Muskat problem  \eqref{PB} 
was established before only in the particular case when the fluids have equal viscosities, cf. \cite{BM21x, CG10}.
Related to \eqref{PB} we mention also the references \cite{EMM12a, EMW18} where a similar scenario is considered, but the upper fluid is taken to be air at uniform pressure.
The question whether the two interfaces can come into contact has been also addressed in the literature in the particular setting of fluids with equal viscosities.
 Namely, for solutions which are not global but bounded in ${\rm C}^{1+\gamma}(\R^2)$, $\gamma\in(0,1)$, squirt singularities (the fluid interfaces touch along a curve segment) were excluded in
    \cite{CG10, BM21x}.  
 Moreover, uniform bounds on the curvature of the interfaces prevent also  the formation of splash singularities, that is single point collisions of the  interfaces, see \cite{GS14}.
The situation is  different in the framework of the  one-phase Muskat problem where splash singularities are one of
 the blow-up mechanisms, see \cite{CCFG16, CP18}, while squirt singularities cannot occur \cite{CP17, CP18}.

In contrast, the classical  two-phase Muskat problem  has been studied recently quite intensively, the results ranging from well-posedness, cf. e.g. \cite{AN2021, MBV18, MBV19, AL20, BCS16, BCG14, CCG11, CG07, FN2021, NP20},  to  global existence results, cf. e.g.\cite{MBV19, Cam19, CCGS13, CGCSS14x, CGSV17, DLL17, GGPS19, PS17}, and to results on the  formation of singularities, cf. e.g. \cite{CCFG13, GG14, CCFGL12, CCF11}.
We point out that the space $H^{3/2}(\R)$ is a critical space for the two-phase Muskat problem \cite{AM21x, NP20} and  also for \eqref{PB}.

In order to  establish our main result, see Theorem~\ref{MT1}, where we prove in particular that \eqref{PB} is locally well-posed in all subcritical $L_2$-Sobolev spaces, we  use the same strategy as for the two-phase flow \cite{MBV18, MBV20} and reexpress \eqref{PB} as an evolution problem
 for the interfaces between the fluids, see Section~\ref{Sec:3}. 
 This problem is nonlinear, strongly coupled, and  of parabolic type in a regime where the Rayleigh-Taylor condition is satisfied.
 Due to the strong coupling we encounter new difficulties compared to the analysis in \cite{MBV18, MBV20} which we outline below when introducing our main result. 

To be more precise, we first introduce some notation.
Let $r\in(3/2,2)$ be fixed and set
\[
\cO_r\coloneqq\{X\coloneqq(f,h)\in H^{r}(\R)^2\,:\, f+c_\infty>h\}, 
\] 
which is an open subset of $H^{r}(\R)^2$. 
In order to reformulate~\eqref{PB} as an evolution problem for the pair $X:=(f,h)\in \cO_r$ we first prove in Theorem~\ref{T:1} that the velocity is determined at each time instant by~$X$. 
This property is equivalent  to establishing  the unique solvability of the equation
\begin{equation}\label{RESint}
(1-A_\mu\A(X))[\oo]=\Theta X',
\end{equation}
cf. Theorem~\ref{T:1}, where $\A(X)$, see \eqref{OpAB'}, is the adjoint of the double layer potential for Laplace's equation associated to the hypersurface ${\Gamma_h\cup\Gamma_{f}^{c_\infty}}$,
\begin{equation}\label{RESint'}
A_\mu\coloneqq{\rm diag\,}(a_\mu^1,a_\mu^2)\qquad\text{and}\qquad \Theta\coloneqq{\rm diag\,}(\Theta_1,\Theta_2),
\end{equation}
and the entries of these diagonal matrices are given by
\begin{equation}\label{RESint''}
\Theta_1:=\cfrac{(\rho_1-\rho_2)gk}{\mu_1+\mu_2},\quad\Theta_2:=\cfrac{(\rho_2-\rho_3)gk}{\mu_2+\mu_3},\qquad a_\mu^1\coloneqq\cfrac{\mu_1-\mu_2}{\mu_1+\mu_2},\qquad a_\mu^2\coloneqq\cfrac{\mu_2-\mu_3}{\mu_2+\mu_3}.
\end{equation}
The unknown $\oo$ in \eqref{RESint} is related to the jump of the velocities at each interface, see the proof of Theorem~\ref{T:1}.
The solvability of \eqref{RESint}  is addressed in Theorem~\ref{Tinvert} and  is  a major departure from to the two-phase setting because $\A(X)$ appears here multiplied by a matrix which has in general different entries
(with possibly opposite sign), and~\eqref{RESint} cannot be reformulated as a problem related to the spectrum of $\A(X)$ as in the two-phase setting. 
In Section~\ref{Sec:2} we identify a (maximal) unbounded open subset $\cU_r$ of $\cO_r$ with the property that for each~$X\in \cU_r$, the equation~\eqref{RESint} is uniquely solvable in~$H^{r-1}(\R)^2.$ 
As a special feature of the multiphase Muskat problem, 
the underlying Rellich identities used to prove  this result in the case when the viscosities are ordered,  that is~${a_\mu^1a_\mu^2>0}$ (in this case $\cU_r$ and $\cO_r$ coincide), 
  do not provide useful estimates when ${a_\mu^1a_\mu^2<0}$ and  in this situation a different Neumann series type argument is employed. 

These facts enable us to reformulate the multiphase Muskat problem~\eqref{PB} as the nonlinear and nonlocal autonomous  evolution equation for $X$ of the form
\begin{align}\label{NNEP*}
\frac{dX(t)}{dt}=\Phi(X(t)),\quad t\geq0,\qquad X(0)=X_0:=(f_0,h_0),
\end{align}
where $\Phi:=(\Phi_1,\Phi_2):\cU_r\subset H^{s}(\R)^2\to H^{s-1}(\R)^2$ is a smooth map, see Section~\ref{Sec:3}.
Therefore we will also refer only to $X$ as being the solution to \eqref{PB}.
The well-known Rayleigh-Taylor condition, see~\eqref{RT}, which is a sign restriction on the jump of the pressure gradients in normal direction at each interface,
is reexpressed in our context by the relations
\begin{equation*}
\Theta_1+a_\mu^1\Phi_1(X)<0 \qquad\text{and}\qquad \Theta_2+a_\mu^2\Phi_2(X)<0,
\end{equation*}
see \eqref{RTRef}.
In the stable regime considered herein, cf. \eqref{stre}, the constants $\Theta_i$, $i=1,\,2$, are both negative, and the Rayleigh-Taylor condition identifies the open subset 
\begin{equation*} 
\cV_r:=\big\{X\in\cU_r\,:\, \text{$\Theta_1+a_\mu^1\Phi_1(X)<0 $ and $\Theta_2+a_\mu^2\Phi_2(X)<0$}\big\}
\end{equation*} 
of $\cU_r$. 
The core of our analysis in Section~\ref{Sec:3} is devoted to showing that the evolution problem~\eqref{NNEP*} is of parabolic type in $\cV_r$.
Having established this property, we can use the abstract parabolic theory from \cite{L95} to prove our main result.

\begin{thm}\label{MT1}
Let $r\in(3/2,2)$.
Given $X_0\in\cV_r$, the multiphase Muskat problem  \eqref{PB} has a unique maximal solution~${X\coloneqq{} X(\,\cdot\,; X_0)}$  such that
\[ X\in {\rm C}([0,T^+),\cV_r)\cap {\rm C}^1([0,T^+),  H^{r-1}(\R)^2)\] 
and the associated velocities and pressures satisfy for each $0\leq t< T^+$
\begin{itemize}
\item $v_i(t)\in {\rm BUC}(\0_i(t))\cap {\rm C}^\infty(\0_i(t)),\, p_i(t)\in {\rm UC}^1(\0_i(t))\cap {\rm C}^\infty(\0_i(t))$ for $1\leq i\leq 3,$\\[-1ex]
\item $[x\mapsto v_i(t,x,c_\infty+f(t,x))]\in H^{r-1}(\R)^2$ for  $1\leq i\leq 2,$\\[-1ex]
\item $[x\mapsto v_i(t,x,h(t,x))]\in H^{r-1}(\R)^2 $ for $ 2\leq i\leq 3,$
\end{itemize}
where $T^+=T^+(X_0)\in(0,\infty]$ denotes the maximal time of existence.\footnote{Given $\0\subset\R^2$, we denote by ${\rm BUC}(\0)$ the Banach space of bounded and uniformly continuous functions
 and~${{\rm UC}^1(\0)}$  is the set of functions with uniformly continuous first order derivatives.}
Moreover, we have:
\begin{itemize}
\item[(i)] The solution depends continuously on the initial data;
\item[(ii)]  $X\in {\rm C}^\infty((0,T^+ )\times\R,\R^2)\cap {\rm C}^\infty ((0,T^+),  H^k(\R)^2)$ for each $k\in\N$.
\end{itemize}
 \end{thm}

\section{Unique solvability of the fixed time problem}\label{Sec:2}

The main goal of this section is to prove that  $X=(f,h)$ identifies at each time  instant the velocity and the pressure in the bulk, and to provide  an explicit formula for the velocity.
More precisely, we  address the unique solvability of the system \eqref{eq:S1}-\eqref{eq:S3}, see Theorem~\ref{T:1} below, where we show that the velocity 
can be expressed by using singular integrals with kernels depending $X$ and with a density vector  $\oo$ which is the unique solution to \eqref{RESint}.
We first address the unique solvability of \eqref{RESint}, see Theorem~\ref{Tinvert} below.

Given $u\in W^1_\infty(\R)$, we set
\begin{equation}\label{notationa}
\begin{aligned}
\bA(u)[\oo](x)&\coloneqq \frac{1}{\pi}\PV\int_\R\frac{su'(x)-\delta_{[x,s]}u}{s^2+(\delta_{[x,s]}u)^2}\oo(x-s)\, ds,\\
\bB(u)[\oo](x)&\coloneqq\frac{1}{\pi}\PV\int_\R\frac{s+u'(x)\delta_{[x,s]}u}{s^2+(\delta_{[x,s]}u)^2}\oo(x-s)\, ds,
\end{aligned}
\end{equation}
where, as usual, $\PV$ is the principal value.
 
Moreover, given  $X:=(f,h)\in W^1_\infty(\R)^2 $ with the property that~${\inf(c_\infty+f-h)>0},$ we define
\begin{equation*}
\begin{aligned}
S(X)[\oo](x)&\coloneqq \frac{1}{\pi}\int_\R\frac{sf'(x)-\delta_{[x,s]}X}{s^2+(\delta_{[x,s]}X)^2}\oo(x-s)\, ds,\\
S'(X)[\oo](x)&\coloneqq\frac{1}{\pi}\int_\R\frac{sh'(x)-\delta'_{[x,s]}X}{s^2+(\delta'_{[x,s]}X)^2}\oo(x-s)\, ds,\\
T(X)[\oo](x)&:=\frac{1}{\pi}\int_\R\frac{s+f'(x)\delta_{[x,s]}X}{s^2+(\delta_{[x,s]}X)^2}\oo(x-s)\, ds,\\
T'(X)[\oo](x)&:=\frac{1}{\pi}\int_\R\frac{s+h'(x)\delta'_{[x,s]}X}{s^2+(\delta'_{[x,s]}X)^2}\oo(x-s)\, ds,
\end{aligned}
\end{equation*}
In the above definitions  we have used the shorthand notation 
\begin{equation}\label{notat}
\begin{aligned}
&\delta_{[x,s]}u\coloneqq  u(x)-u(x-s),\\[1ex]
& \delta_{[x,s]}X\coloneqq  c_\infty+f(x)-h(x-s),\\[1ex]
& \delta'_{[x,s]}X\coloneqq  h(x)-c_\infty-f(x-s).
\end{aligned}
\end{equation}
The operators $\bA(u)$ and $\bB(u)$ have been introduced in the context of the Muskat problem in \cite{MBV18}.
In particular, we have $\bA(0)=0 $ and $\bB(0)=H$, where $H$ is the Hilbert transform. 
Let $r\in(3/2,2)$ be fixed and let $\cO_r$ be the open subset of $H^{r}(\R)^2$ defined in Introduction.
As shown in \cite{MBV18, AM21x}, it holds that   
\begin{equation}\label{regAB}
\bA(u),\, \bB(u)\in\kL(L_2(\R))\cap \kL(H^{r-1}(\R)),\qquad u\in H^{r}(\R),
\end{equation}
 and
 \begin{equation}\label{regAB'}
[u\mapsto\bA(u)],\, [u\mapsto \bB(u)]\in {\rm C}^{1-}(W^1_\infty(\R),\kL(L_2(\R))).
\end{equation}
Moreover, it is proven in \cite[Theorem 3.5]{MBV18} and \cite[Theorem 5]{AM21x}   that
\begin{equation}\label{isom}
\begin{aligned}
&\lambda-\bA(u) \in{\rm Isom }(L_2(\R)), \qquad u\in {\rm C}^1(\R),\\[1ex]
&\lambda-\bA(u)\in {\rm Isom }(H^{r-1}(\R)),\qquad u\in H^{r}(\R)
\end{aligned}
\end{equation}
for all  $\lambda\in\R$ with $|\lambda|\geq1$. 
Given  Banach spaces $\mathbb{E},\, \mathbb{F}$, we denote by ${\rm Isom }(\mathbb{E},\mathbb{F})$ the open subset of~$\mathcal{L}(\mathbb{E}, \mathbb{F})$ that consists of  invertible operators.

With respect to the operators $S(X)$, $S'(X)$, $T(X)$, and $T'(X)$  we infer from the arguments in the proof of \cite[Lemma~2.2, Lemma~2.3 and Corollary~2.4]{BM21x} that  
\begin{equation}\label{SSS'}
S,\, S', \, T,\ T'\in {\rm C}^{1-}(\{X\in W^1_\infty(\R)^2\,:\, \inf(c_\infty+f-h)>0\},\kL(L_2(\R))).
\end{equation}
Moreover, it is shown in  \cite[Lemma~2.5]{BM21x} that
\begin{equation}\label{SSS}
S(X),\, S'(X),\, T(X),\, T'(X)\in\kL(L_2(\R), H^{1}(\R)),\quad X\in\cO_r.
\end{equation}

Given  $X\in\cO_r$, let  
\begin{equation*}
\A(X):=(\A_1(X),\,\A_2(X)) \qquad \text{and} \qquad \B(X):=(\B_1(X),\,\B_2(X))
\end{equation*}
 denote the linear operators given by
\begin{equation}\label{OpAB}
\begin{aligned}
\A_1(X)[\oo]:=\bA(f)[\oo_1]+S(X)[\oo_2],\qquad \A_2(X)[\oo](x):=S'(X)[\oo_1]+\bA(h)[\oo_2],\\[1ex]
\B_1(X)[\oo]:=\bB(f)[\oo_1]+T(X)[\oo_2],\qquad \B_2(X)[\oo](x):=T'(X)[\oo_1]+\bB(h)[\oo_2],\\[1ex]
\end{aligned}
\end{equation}  
for $\oo\coloneqq(\oo_1,\oo_2)\in L_2(\R)^2$, or  equivalently
\begin{equation}\label{OpAB'}
\A(X):=
\begin{pmatrix}
\bA(f)& S(X)\\[1ex]
S'(X)&\bA(h)
\end{pmatrix}
\qquad \text{and} \qquad
\B(X):=
\begin{pmatrix}
\bB(f)& T(X)\\[1ex]
T'(X)&\bB(h)
\end{pmatrix}.
\end{equation} 
It directly follows from \eqref{regAB} and \eqref{SSS} that 
\begin{equation}\label{regcab}
\A(X),\, \B(X)\in\kL(L_2(\R)^2)\cap \kL(H^{r-1}(\R)^2),\qquad X\in\cO_r,
\end{equation}
while 
\eqref{regAB'} and \eqref{SSS'} imply that
\begin{equation}\label{regcab'}
[X\mapsto \A(X)],\, [X\mapsto \B(X)]\in{\rm C^{1-}}(\{X\in W^1_\infty(\R)^2\,:\, \inf(c_\infty+f-h)\},\kL(L_2(\R)^2)).
\end{equation}

The main step  in the proof of  Theorem~\ref{T:1} below is the following result.

\begin{thm}\label{Tinvert}
Let $r\in(3/2,2)$ and $A\coloneqq{\rm diag\,}(a_1,a_2)\in\R^{2\times 2}$ with $\max\{|a_1|,\, |a_2|\}< 1$ be given.
Then, $\cU_r:=\cU\cap H^{r}(\R)^2$, where 
\begin{equation*}
\cU\coloneqq\{X\in {\rm C}^1(\R)^2\,:\, \text{$\inf(c_\infty+f-h)>0$ and $1-T_A(X)\in{\rm Isom}(L_2(\R))$}\}
\end{equation*}
and
\[
T_A(X):=a_1a_2(1-a_1\bA(f))^{-1}S(X)(1-a_2\bA(h))^{-1}S'(X),
\] 
is a nonempty  open subset of $\cO_r$ and $1-A\A(X)\in{\rm Isom}(H^{r-1}(\R)^2)$ for all~${X\in\cU_r}$.
Moreover, it holds that: 
\begin{itemize}
\item[(a)] If $a_1a_2\geq 0$, then $\cU_r=\cO_r$; 
\item[(b)]  If $a_1a_2< 0$, there exists a constant $\sigma=\sigma(A)>0$ such that 
 $$\{X\in\cO_r\,:\, \|X\|_{W^1_\infty}<\sigma\} \subset\cU_r.$$
\end{itemize}
In particular,  $\cU_r$ is an unbounded subset of $H^{r}(\R)^2.$
\end{thm}

Before proceeding with the proof we add some comments.

\begin{rem}\label{R:1} \phantom{a}

\begin{itemize}
\item[(i)]
The operator $\bA(f)$ is the adjoint of the double layer potential for Laplace's equation associated to  the hypersurface $\Gamma_f^{c_\infty}$ 
and the operator $\A(X)$ its generalization to the hypersurface~${\Gamma_h\cup\Gamma_{f}^{c_\infty}}.$
As already mentioned, $\lambda-\bA(f)$ (and $\lambda-\A(X)$,  cf. Corollary~\ref{Cor1} below) is an $L_2$-isomorphism for all $\lambda\in\R$ with $|\lambda|\geq1$.
 This property is essential in the study of the two-phase Muskat problem, see \cite{MBV18, MBV20, CCG11}, but in the context of the multiphase Muskat problem we need to consider 
 the invertibility of the operator~${1-A_\mu\A(X)}$,  see Theorem~\ref{T:1}, which is not a spectral problem for $\A(X)$.

On the basis of the Rellich identities \eqref{For1} we  show in Theorem~\ref{Tinvert} that~${1-A\A(X)}$ is an isomorphism whenever the entries of the matrix $A$ have the same sign. 
If $a_1a_2<0$, we can invert this operator only under a smallness assumption on the $W^1_\infty-$norm of $X$, and it is not clear to us whether this smallness assumption can be dropped. 
In fact~${1-A\A(X)}$  is not invertible if $A\coloneqq{\rm diag\,}(1,-1)$ and $X=0$, see~(ii) . 
\item[(ii)] Let $A\coloneqq{\rm diag\,}(1,-1)$. Then $1-A\A(0)$  is not invertible in~$\kL(L_2(\R)^2)$ (or~${\kL(H^{r-1}(\R)^2)}$).

Indeed, let $F=(F_1,0)\in H^{r-1}(\R)^2$ be chosen such that~${\cF F_1\in {\rm C}^{\infty}_0(\R)}$ is a function which satisfies $\cF F_1(\xi)=1$ for all $|\xi|\leq 1 $ (where $\cF$ denotes the Fourier transform).
If~$\oo\coloneqq(\oo_1,\oo_2)\in L_2(\R)^2$ satisfies $(1-A\A(0))[\oo]=F$, we arrive in virtue of $\bA(0)=0$ and of $S'(0)=-S(0) $ with $S'(0)[\oo_2]=\varphi*\oo_2$, where
\[
\varphi(s)\coloneqq\frac{1}{\pi}\frac{c_\infty}{s^2+c_\infty^2},\qquad s\in\R,
\]
 at the system
 \[
\left.
\begin{array}{llll}
\oo_1-\varphi*\oo_2=F_1,\\[1ex]
\oo_2-\varphi*\oo_1=0.
\end{array}
\right\} 
 \]
 Applying the Fourier transform on both equations and using the convolution theorem,  we get 
 \begin{equation}\label{Formm}
\cF\oo_1=\frac{\cF F_1}{1-(\sqrt{2\pi}\cF\varphi)^2}. 
 \end{equation}
 The function $\cF\varphi$ can be computed explicitly. Indeed, let $\psi(s):=e^{-c_\infty|s|}$, $s\in\R$.
 It then holds 
 \[
\sqrt{2\pi}\cF\psi(\xi)=\int_\R e^{-c_\infty|s|}e^{-is\xi}\, ds=\frac{e^{-(c_\infty+i\xi)s}}{-(c_\infty+i\xi)}\Big|_0^\infty
+\frac{e^{(c_\infty-i\xi)s}}{(c_\infty-i\xi)}\Big|_{-\infty}^0=2\pi\varphi(\xi)
 \]
 for $\xi\in\R$, hence $\sqrt{2\pi}\cF\varphi=\psi$.
 Because $(1-\psi^2(\xi))/|\xi|\to 2c_\infty$ for $\xi\to0$, we conclude that the function $\cF\oo_1$ defined in \eqref{Formm} does not belong to $L_2(\R)$.
  Therefore,~${1-A\A(0)}$ is not surjective, and this proves the claim. 
\end{itemize}
\end{rem}

We are now in a position to prove Theorem~\ref{Tinvert}.
\begin{proof}[Proof of Theorem~\ref{Tinvert}] 
Let $X\in\cO_r$. Recalling~\eqref{OpAB'}, the invertibility of the operator~${1-A\A(X)}$ in~${\kL(H^{r-1}(\R)^2)}$ is equivalent to the unique solvability of the system
\begin{equation}\label{SYSt}
\left.
\begin{aligned}
&\oo_1-a_1\bA(f)[\oo_1]-a_1S(X)[\oo_2]=F_1,\\[1ex]
&\oo_2-a_2S'(X)[\oo_1]-a_2\bA(h)[\oo_2]=F_2
\end{aligned}
\right\}
\end{equation}
in $H^{r-1}(\R)^2$ for each  $F\coloneqq(F_1,F_2)\in H^{r-1}(\R)^2$.
Since
$1-a_1\bA(f) $ and~${1-a_2\bA(h)}$ are invertible   in $ \kL(H^{r-1}(\R))$, cf. $\eqref{isom}_2$,
this property together with \eqref{SSS}, enables us to conclude that the system \eqref{SYSt} is equivalent to the following equation for $\oo_1$ in $H^{r-1}(\R)$:
   \begin{equation}\label{Eqq}
   (1-T_A(X))[\oo_1]=(1-a_1\bA(f))^{-1}[F_1+a_1S(X)(1-a_2\bA(h))^{-1}[F_2]].
   \end{equation}
Hence, if $X\in\cU_r$, then \eqref{Eqq} has a unique solution $\oo_1\in L_2(\R).$
 Additionally, since the right side of \eqref{Eqq} and $T_A(X)[\oo_1]$  belong to $H^{r-1}(\R)$, cf.  $\eqref{isom}_2$  and \eqref{SSS},
 it follows from~\eqref{Eqq} that actually $\oo_1\in H^{r-1}(\R)$. 
 Hence, $1-A\A(X)\in{\rm Isom}(H^{r-1}(\R)^2)$ for all $X\in\cU_r$.  
 
That $\cU_r$ is an open subset of $\cO_r$ is a direct consequence of  the property
\begin{equation}\label{reggg}
[X\mapsto T_A(X)]\in{\rm C}^{1-}(\{X\in {\rm C}^1(\R)^2\,:\, \inf(c_\infty+f-h)>0\},\kL(L_2(\R))).
\end{equation}
The property \eqref{reggg}  is a direct consequence of \eqref{regAB'}, $\eqref{isom}_1$,  \eqref{SSS'}, and of the smoothness of the mapping $[T\mapsto T^{-1}]:{\rm Isom}(L_2(\R))\to {\rm Isom}(L_2(\R)).$  

In order to show that $\cU_r$ is not empty, we observe that  $T_A(0)=a_1a_2S(0)S'(0)$. 
Recalling that~${S'(0)=-S(0)}$ and $S'(0)[\oo_2]=\varphi*\oo_2,$  see Remark~\ref{R:1}~(ii),
we have
 \begin{align*}
  \|S(0) \|_{\kL(L_2(\R))}=\|S'(0) \|_{\kL(L_2(\R))}\leq 1,
 \end{align*}
hence $\|T_A(0)\|_{\kL(L_2(\R))}\leq |a_1a_2|<1.$ 
 Consequently $0\in\cU_r$ and, recalling \eqref{reggg}, we may conclude there exits  a constant $\sigma=\sigma(A)>0$ such that 
 $\{X\in\cO_r\,:\, \|X\|_{W^1_\infty}<\sigma\} \subset\cU_r.$\medskip

It remains to show that in the case   (a), that is  when  $a_1a_2\geq0,$ we have $\cU_r=\cO_r$. 
If $a_1a_2=0,$ then $T_A=0$, hence $\cU_r=\cO_r$.
The proof in the case when $a_1a_2>0$ is more involved.
The crucial step is to prove that~$1-A\A(X)\in{\rm Isom}(L_2(\R)^2)$ for all $X\in\cO_r$.
To this end we  next view~$A$ as a parameter matrix, and we prove there exists a constant~${C=C(\|X'\|_\infty)}>0$ such that 
\begin{equation}\label{SYSE}
\|(1-A\A(X))[\oo]\|_2\geq C m(A)\|\oo\|_2\qquad\text{for all $\oo\in L_2(\R)^2$,}
\end{equation}
where
\begin{equation}\label{mA}
m(A):=\min\Big\{\frac{1+a_1}{|a_1|},\, \frac{1-a_2}{|a_2|},\, |a_1|(1-a_1),\, |a_2|(1+a_2) \Big\}>0.
\end{equation}
Having established \eqref{SYSE}, we note that if $a_1 $ and $a_2$ are small, then ${1-A\A(X)\in\kL(L_2(\R)^2)}$ is invertible, and the method of continuity, 
cf. \cite[Proposition I.1.1.1]{Am95}, together with \eqref{SYSE} implies that
 $1-A\A(X)\in{\rm Isom}(L_2(\R)^2)$ for all $A\coloneqq{\rm diag\,}(a_1,a_2)$ with $\max\{|a_1|,\, |a_2|\}< 1$.
 If~${F\in H^{r-1}(\R)^2},$ we may use the fact that $S(X)[\oo_2],\, S'(X)[\oo_1]\in H^{1}(\R)\subset H^{r-1}(\R)$  to conclude from \eqref{SYSt}, in view of~$\eqref{isom}_2$, 
  that~${\oo:=(1-A\A(X))^{-1}[F]}$ belongs to~${H^{r-1}(\R)^2}$. 
This establishes the isomorphism property~${1-A\A(X)\in{\rm Isom}(H^{r-1}(\R)^2)}$ for all $X\in\cO_r$.

We now proceed with the proof of~\eqref{SYSE}.
Given $X=(f,h)\in {\rm C}_0^\infty(\R)^2$ with $c_\infty+f>h$ and~${\oo:=(\oo_1,\oo_2)\in {\rm C}^\infty_0(\R)^2}$, let $v\coloneqq(v^1,v^2)$ be given by
\begin{align*} 
v(z)\coloneqq \frac{1}{\pi}\int_\R\frac{(c_\infty+f(s)-y,x-s)}{(x-s)^2+(y-c_\infty-f(s))^2}\oo_1(s)\, ds+\frac{1}{\pi}\int_\R\frac{(h(s)-y,x-s)}{(x-s)^2+(y-h(s))^2}\oo_2(s)\, ds 
\end{align*}
for $z\coloneqq(x,y)\in\R^2\setminus(\Gamma_h\cup\Gamma_{f}^{c_\infty}).$
We next infer from the results in \cite[Appendix A]{BM21x} that the restrictions~$v_i:=v|_{\0_i}$ belong to~${{\rm BUC}(\0_i)\cap {\rm C}^\infty(\0_i)}$,~${1\leq i\leq 3}$.
Moreover, the normal and tangential traces of $v$ are related to the operators  $\A(X)$ and $\B(X)$ defined in \eqref{OpAB}, since
\begin{align*}
\A_1(X)[\oo](x)&\coloneqq\langle v_i(x,c_\infty+f(x))|(1,f'(x))\rangle+(-1)^{1+i}\oo_1(x),\quad i=1,\,2,\, x\in\R,\\[1ex]
\A_2(X)[\oo](x)&\coloneqq\langle v_i(x,h(x))|(1,h'(x))\rangle+(-1)^i\oo_2(x),\quad i=2,\, 3,\, x\in\R,
\end{align*}
 and
\begin{align*}
\B_1(X)[\oo](x)&\coloneqq\langle v_i(x,c_\infty+f(x))|(-f'(x),1)\rangle,\quad i=1,\,2,\, x\in\R, \\[1ex]
 \B_2(X)[\oo](x)&\coloneqq\langle v_i(x,h(x))|(-h'(x),1)\rangle,\quad i=2,\, 3,\, x\in\R,
\end{align*}
see \eqref{v12rand} and \eqref{v23rand} below.
Additionally, it is a matter of direct computations to verify that  
\begin{equation}\label{rotdiv}
\p_yv_i^1-\p_x v_i^2=\p_xv_i^1+\p_y v_i^2=0\qquad\text{in $\0_i$, $1\leq i\leq 3$}.
\end{equation}
Since  $\oo$ is compactly supported,  it is not difficult to prove that there exists a constant $C>0$ such that  
\begin{equation}\label{quantde}
|v(z)|\leq \frac{C}{|z|} \qquad \text{for $|z|\to\infty,$}
\end{equation}
see, e.g., \cite[Lemma A.8]{JBThesis}.
In view of  \eqref{rotdiv} and \eqref{quantde} we then get
\begin{equation}\label{Sto}
\int_{\0_i}{\rm div}\,\begin{pmatrix}
2v_i^1v_i^2\\[1ex]
(v_i^2)^2-(v_i^1)^2
\end{pmatrix}\, dz=0,\qquad 1\leq i\leq 3.
\end{equation}
 Stokes' theorem  then leads us to the following system of equations
 \begin{equation}\label{For1}
 \left.
 \begin{aligned}
 0&=\int_\R\frac{1}{1+f'^2}\big[(\A_1(X)[\oo]-\oo_1)^2-2f'\B_1(X)[\oo](\A_1(X)[\oo]-\oo_1)-(\B_1(X)[\oo])^2\big]\, dx,	\\[2ex]
0&= \int_\R\frac{1}{1+f'^2}\big[(\A_1(X)[\oo]+\oo_1)^2-2f'\B_1(X)[\oo](\A_1(X)[\oo]+\oo_1)-(\B_1(X)[\oo])^2\big]\, dx\\[1ex]
& \hspace{0,45cm}-\int_\R\frac{1}{1+h'^2}\big[(\A_2(X)[\oo]-\oo_2)^2-2h'\B_2(X)[\oo](\A_2(X)[\oo]-\oo_2)-(\B_2(X)[\oo])^2\big]\, dx,\\[2ex]
 0&=\int_\R\frac{1}{1+h'^2}\big[(\A_2(X)[\oo]+\oo_2)^2-2h'\B_2(X)[\oo](\A_2(X)[\oo]+\oo_2)-(\B_2(X)[\oo])^2\big]\, dx.	
 \end{aligned}\right\}
 \end{equation}
Using the algebraic identity
\begin{equation*}
(\A_i(X)[\oo]\pm\oo_i)^2=\frac{(\oo_i-a_i\A_i(X)[\oo])^2-2(1\pm a_i)\oo_i(\oo_i-a_i\A_i(X)[\oo])+(1\pm a_i)^2\oo_i^2}{a_i^2}
\end{equation*}
which holds for all $i=1,\, 2$, the system \eqref{For1} is equivalent to
\begin{equation*}
 \left.
 \begin{aligned}
 0&=\int_\R\frac{1}{1+f'^2}\Big[\frac{(\oo_1-a_1\A_1(X)[\oo])^2-2(1- a_1)\oo_1( \oo_1-a_1\A_1(X)[\oo])+(1- a_1)^2\oo_1^2}{a_1^2} \\[1ex]
 &\hspace{2.425cm}-2f'\B_1(X)[\oo](\A_1(X)[\oo]-\oo_1)-(\B_1(X)[\oo])^2\Big]\, dx,	\\[2ex]
0&= \int_\R\frac{1}{1+f'^2}\Big[\frac{( \oo_1-a_1\A_1(X)[\oo])^2-2(1+ a_1)\oo_1( \oo_1-a_1\A_1(X)[\oo])+(1+ a_1)^2\oo_1^2}{a_1^2}\\[1ex]
&\hspace{2.425cm}-2f'\B_1(X)[\oo](\A_1(X)[\oo]+\oo_1)-(\B_1(X)[\oo])^2\Big]\, dx\\[1ex]
& \hspace{0,45cm}-\int_\R\frac{1}{1+h'^2}\Big[\frac{( \oo_2-a_2\A_2(X)[\oo])^2-2(1- a_2)\oo_2( \oo_2-a_2\A_2(X)[\oo])+(1- a_2)^2\oo_2^2}{a_2^2} \\[1ex]
 &\hspace{2.875cm}-2h'\B_2(X)[\oo](\A_2(X)[\oo]-\oo_2)-(\B_2(X)[\oo])^2\Big]\, dx,\\[2ex]
 0&=\int_\R\frac{1}{1+h'^2}\Big[\frac{( \oo_2-a_2\A_2(X)[\oo])^2-2(1+ a_2)\oo_2( \oo_2-a_2\A_2(X)[\oo])+(1+ a_2)^2\oo_2^2}{a_2^2} \\[1ex]
 &\hspace{2.425cm}-2h'\B_2(X)[\oo](\A_2(X)[\oo]+\oo_2)-(\B_2(X)[\oo])^2\Big]\, dx.	
 \end{aligned}\right\}
 \end{equation*}
 We next multiply  the first equation  by $(1+a_1),$ the second  by $-(1-a_1),$ and the third  by~${-(1-a_1)(1-a_2)/(1+a_2)},$  and then sum up the resulting identities to arrive, 
 after  multiplying the resulting identity  by~${\rm sign\,}(a_1)={\rm sign\,}(a_2)$, at
\begin{equation*} 
 \begin{aligned}
 &\int_\R\frac{1}{1+f'^2}\Big[\frac{( \oo_1-a_1\A_1(X)[\oo])^2}{|a_1|(1-a_1)} +\frac{2{\rm sign\,}(a_1)}{1-a_1}f'\B_1(X)[\oo]( \oo_1-a_1\A_1(X)[\oo])\Big]\, dx	\\[2ex]
&\hspace{0,45cm}+\int_\R\frac{1}{1+h'^2}\Big[\frac{( \oo_2-a_2\A_2(X)[\oo])^2}{|a_2|(1+a_2)} +\frac{2{\rm sign\,}(a_2)}{1+a_2}h'\B_2(X)[\oo]( \oo_2-a_2\A_2(X)[\oo] )\Big]\, dx\\[2ex]	
&=\int_\R\frac{1}{1+f'^2}\Big[ \frac{1+ a_1}{|a_1|}\oo_1^2 +\frac{|a_1|}{1-a_1}(\B_1(X)[\oo])^2\Big]\, dx	\\[2ex]
& \hspace{0,45cm}+\int_\R\frac{1}{1+h'^2}\Big[\frac{1- a_2}{|a_2|} \oo_2^2 + \frac{|a_2|}{1+a_2} (\B_2(X)[\oo])^2\Big]\, dx.	
 \end{aligned}
 \end{equation*}
 The latter estimate, the H\"older inequality, and Young's inequality imply there exists a positive constant~${C=C(\|X'\|_\infty})$ such that
\begin{equation*}
\frac{\|\oo_1-a_1\A_1(X)[\oo]\|_2^2}{|a_1|(1-a_1)}+ \frac{\|\oo_2-a_2\A_2(X)[\oo]\|_2^2}{|a_2|(1+a_2)}\geq C\Big(\frac{1+a_1}{|a_1|}\|\oo_1\|_2^2+\frac{ 1- a_2}{|a_2|}\|\oo_2\|_2^2\Big).
\end{equation*}
The claim~\eqref{SYSE} now follows from the latter estimate, by using \eqref{regcab'} and  a standard density argument. 
\end{proof}

We now conclude with a result on the resolvent of $\A(X).$

\begin{cor}\label{Cor1}
Let $r\in(3/2,2)$ and $X\in \cO_r$ be given. Then 
\[
\lambda-\A(X)\in {\rm Isom}(L_2(\R)^2)\cap {\rm Isom}(H^{r-1}(\R)^2)\qquad\text{for all $\lambda\in\R\setminus(-1,1).$}
\] 
\end{cor}
\begin{proof}
The claim follows for  $|\lambda|>1$  from Theorem~\ref{Tinvert} (by choosing $a_1=a_2=1/\lambda$).
It remains to establish the result for $\lambda\in\{\pm 1\}.$
To this end we infer from \eqref{For1}, by using Young's and Hölder's inequalities, there exists a   constant~${C=C(\|X'\|_\infty)}\geq1 $ such that
\begin{equation*}
\begin{aligned}
&C^{-1}\|\oo_1-\A_1(X)[\oo]\|_2\leq \|\B_1(X)[\oo]\|_2\leq C \|\oo_1-\A_1(X)[\oo]\|_2,\\[1ex]
&C^{-1}\|\oo_2+\A_2(X)[\oo]\|_2\leq \|\B_2(X)[\oo]\|_2\leq C \|\oo_2+\A_2(X)[\oo]\|_2,\\[1ex]
&C^{-1}\big(\|\oo_1+\A_1(X)[\oo]\|_2+\|\B_2(X)[\oo]\|_2\big)\leq\|\oo_2-\A_2(X)[\oo]\|_2+ \|\B_1(X)[\oo]\|_2,\\[1ex]
&\|\oo_2-\A_2(X)[\oo]\|_2+ \|\B_1(X)[\oo]\|_2\leq C\big(\|\oo_1+\A_1(X)[\oo]\|_2+\|\B_2(X)[\oo]\|_2\big).
\end{aligned}
\end{equation*}
Using these relations, we get
\begin{align*}
\|\oo\|_2&\leq \|\oo_1-\A_1(X)[\oo]\|_2+\|\oo_1+\A_1(X)[\oo]\|_2+\|\oo_2-\A_2(X)[\oo]\|_2+\|\oo_2+\A_2(X)[\oo]\|_2\\[1ex]
&\leq C\big(|\B_1(X)[\oo]\|_2+\|\oo_1+\A_1(X)[\oo]\|_2+\|\oo_2-\A_2(X)[\oo]\|_2+\|\B_2(X)[\oo]\|_2\big)\\[1ex]
&\leq C\min\{\|\oo_2-\A_2(X)[\oo]\|_2+|\B_1(X)[\oo]\|_2,\|\oo_1+\A_1(X)[\oo]\|_2+\|\B_2(X)[\oo]\|_2 \}\\[1ex]
&\leq  C \|(1\pm\A(X))[\oo]\|_2.
\end{align*}
The latter estimates combined with \eqref{SYSE},
 the continuity property \eqref{regcab'}, the method of continuity \cite[Proposition I.1.1.1]{Am95}, and the arguments in the proof of Theorem~\ref{Tinvert}
  enable us to deduce that the claim indeed  holds also for $\lambda\in\{\pm 1\}.$
\end{proof}

We conclude this section   by establishing the unique solvability of the system~\eqref{eq:Stat} when assuming~$X\in \cU_r$, where~$\cU_r$ is the open subset of~$H^r(\R)^2$ found in Theorem~\ref{Tinvert} 
when choosing $A:=A_\mu$.
This shows in particular that for classical solutions to \eqref{PB} (in the sense of Theorem~\ref{MT1}),  the free interfaces identify at each time instant the velocities and the pressures.

\begin{thm}\label{T:1} Let $r\in(3/2,2)$ and $X=(f,h)\in\cU_r$, where $\cU_r$ is the open subset of~$H^r(\R)^2$ found in Theorem~\ref{Tinvert} for the choice 
$A\coloneqq A_\mu\coloneqq{\rm diag\,}(a_\mu^1,a_\mu^2)$, see \eqref{RESint'}-\eqref{RESint''}.
Then the boundary value  problem
\begin{equation}\label{eq:Stat}
\left.\begin{array}{rllllll}
 v_i\!\!\!\!&=&\!\!\!\!-\displaystyle\frac{k}{\mu_i}\big(\nabla p_i+(0,\rho_i g)\big)&\text{in $ \0_i$, $1\leq i\leq 3$},\\[1ex]
{\rm div}\,  v_i\!\!\!\!&=&\!\!\!\!0 &\text{in $ \0_i$, $1\leq i\leq 3$}, \\[1ex]
p_i\!\!\!\!&=&\!\!\!\!p_{i+1}&\text{on $\p\0_i\cap\p\0_{i+1}$, $i=1,\, 2$}, \\[1ex]
 \langle v_i| \nu_i\rangle\!\!\!\!&=&\!\!\!\!  \langle v_{i+1}| \nu_i\rangle &\text{on $\p\0_i\cap\p\0_{i+1}$, $i=1,\, 2$,}\\[1ex]
v_i(x,y)\!\!\!\!&\to&\!\!\!\! 0 &\text{for  $|(x,y)|\to\infty$,  $1\leq i\leq 3$} 
\end{array}\right\}
\end{equation}
has a unique solution\footnote{The pressures $p_1,\, p_2,\, p_3$ are unique only up to the same additive constant.} $(v_1,v_2,v_3,p_1,p_2,p_3)$ such that
\begin{itemize}
\item$v_i\in {\rm BUC}(\0_i)\cap {\rm C}^\infty(\0_i),\, p_i\in {\rm UC}^1(\0_i)\cap {\rm C}^\infty(\0_i)$ for $1\leq i\leq 3,$\\[-1ex]
\item $[x\mapsto v_i(x,c_\infty+f(x))]\in H^{r-1}(\R)$ for  $1\leq i\leq 2,$\\[-1ex]
\item $[x\mapsto v_i(x,h(x))]\in H^{r-1}(\R) $ for $ 2\leq i\leq 3.$
\end{itemize}
Moreover,  setting $v\coloneqq v_1{\bf 1}_{\0_1}+v_2{\bf 1}_{\0_2}+v_3{\bf 1}_{\0_3}$, it holds for $z\coloneqq{}(x,y)\in\R^2\setminus(\G_h\cup\Gamma_{f}^{c_\infty})$ that
\begin{align}\label{forvelo}
\hspace{-0.25cm}v(z)=\frac{1}{\pi}\int_\R\frac{(c_\infty+f(s)-y,x-s)}{(x-s)^2+(y-c_\infty-f(s))^2}\oo_1(s)\, ds+\frac{1}{\pi}\int_\R\frac{(h(s)-y,x-s)}{(x-s)^2+(y-h(s))^2}\oo_2(s)\, ds,
\end{align}
where $\oo:=(\oo_1,\oo_2)\in H^{r-1}(\R)^2$ denotes the unique solution to the equation
\begin{equation}\label{RES?}
(1-A_\mu\A(X))[\oo]=\Theta X' 
\end{equation}
with  $\Theta:={\rm diag\,}(\Theta_1,\Theta_2)$  defined in \eqref{RESint'}-\eqref{RESint''}.
\end{thm}
\begin{proof} We devise the proof in two steps.\medskip

\noindent{Existence.} To each  pair   $\oo=(\oo_1,\oo_2)\in H^{r-1}(\R)^2$ we associate the velocity  $v\coloneqq v(X)[\oo]$ which is defined by \eqref{forvelo} in~${\R^2\setminus(\G_h\cup\Gamma_{f}^{c_\infty}).}$ 
 The results established in \cite[Appendix~A]{BM21x} imply that~${v_i:=v|_{\0_i}\in {\rm BUC}^{r-3/2}(\0_i)}\cap{\rm C}^\infty(\0_i)$ for~${1\leq i\leq 3},$
\begin{align}
v_i(x,c_\infty+f(x))&=\frac{1}{\pi}\PV\int_\R\frac{(-\delta_{[x,s]}f,s)}{s^2+(\delta_{[x,s]}f)^2}\oo_1(x-s)\, ds
+\frac{1}{\pi}\int_\R\frac{(-\delta_{[x,s]}X,s)}{s^2+(\delta_{[x,s]}X)^2}\oo_2(x-s)\, ds\nonumber\\[1ex]
&\hspace{0.5cm}+(-1)^{i}\frac{\oo_1(1,f')}{1+f'^2}(x),\quad i=1,\, 2, \, x\in\R,\label{v12rand}
\end{align}
and
\begin{align}
v_i(x,h(x))&=\frac{1}{\pi}\int_\R\frac{(-\delta'_{[x,s]}X,s)}{s^2+(\delta'_{[x,s]}X)^2}\oo_1(x-s)\, ds
+\frac{1}{\pi}\PV\int_\R\frac{(-\delta_{[x,s]}h,s)}{s^2+(\delta_{[x,s]}h)^2} \oo_2(x-s)\, ds\nonumber\\[1ex]
&\hspace{0.5cm}+(-1)^{i+1}\frac{\oo_2(1,h')}{1+h'^2}(x),\quad i=2,\, 3,\, x\in\R.\label{v23rand}
\end{align}
As a direct consequence of the relations \eqref{v12rand} and \eqref{v23rand} we obtain that  $\eqref{eq:Stat}_4$ holds true.
 It is not difficult to infer from \cite[Lemma 2.1 and Lemma 2.5]{BM21x} that
\begin{itemize}
\item $[x\mapsto v_i(x,c_\infty+f(x))]\in H^{r-1}(\R)^2$ for  $1\leq i\leq 2,$\\[-1ex]
\item $[x\mapsto v_i(x,h(x))]\in H^{r-1}(\R)^2 $ for $ 2\leq i\leq 3.$
\end{itemize}
Moreover, \cite[Lemma A.4]{BM21x},  implies that $\eqref{eq:Stat}_2$ and~$\eqref{eq:Stat}_5$ are satisfied and that
\begin{equation}\label{rotfree}
 \p_yv^1=\p_xv^2 \qquad\text{in  $\R^2\setminus(\G_h\cup\Gamma_{f}^{c_\infty})$}.
\end{equation}

Corresponding to $v$, we now define pressures $p_i:\0_i\to\R$, $1\leq i\leq 3,$ by the formula
\begin{equation}\label{pressures}
p_i(z)\coloneqq-\frac{\mu_i}{k}\Big(\int_0^x\langle v_i(s,d_i(s))|(1,d_i'(s))\rangle\, ds+\int_{d_i(x)}^yv_i^2(x,s)\, ds\Big)-\rho_igy+c_i
\end{equation}
for $z:=(x,y)\in\0_i,$
where $v_i\eqqcolon(v_i^1,v_i^2)$, $c_i\in\R$ is a constant, and  with
\[
d_1\coloneqq{}\|f\|_\infty+c_\infty+1,\qquad d_2\coloneqq{}\frac{1}{2}(c_\infty+f+h),\qquad d_3\coloneqq{}-\|h\|_\infty-1. 
\]
Using \eqref{rotfree}, we deduce that~${p_i\in{\rm  C}^1(\0_i)}$  and that~$\eqref{eq:Stat}_1$ is satisfied.
The regularity properties established  for $v_i$  together with~$\eqref{eq:Stat}_1$ show that~${p_i\in {\rm UC}^1(\0_i)\cap {\rm C}^\infty(\0_i)}$, $1\leq i\leq 3$.

We point out that all equations constituting \eqref{eq:Stat}, excepting $\eqref{eq:Stat}_3$, are valid for any choice of~$\oo$.
We now prove that  the dynamic boundary condition $\eqref{eq:Stat}_3$ identifies $\oo$ as the unique solution to \eqref{RES?}.
To this end we take advantage of $\eqref{eq:Stat}_1$ and \eqref{v12rand}-\eqref{v23rand} and compute that
\begin{align*}
\frac{d}{dx}\big((p_2-p_1)(x,c_\infty+f(x))\big)&=(\rho_1-\rho_2)gf'(x)+\Big\langle\frac{\mu_1v_1-\mu_2v_2}{k}(x,c_\infty+f(x))\Big|(1,f'(x))\Big\rangle\\[1ex]
&=(\rho_1-\rho_2)gf'(x)-\frac{\mu_1+\mu_2}{k}\oo_1(x)+\frac{\mu_1-\mu_2}{k}\A_1(X)[\oo](x)\\[1ex]
\frac{d}{dx}\big((p_3-p_2)(x,h(x))\big)&=(\rho_2-\rho_3)gh'(x)+\Big\langle\frac{\mu_2v_2-\mu_3v_3}{k}(x,h(x)))\Big|(1,h'(x))\Big\rangle\\[1ex]
&=(\rho_2-\rho_3)gh'(x)-\frac{\mu_2+\mu_3}{k}\oo_2(x)+\frac{\mu_2-\mu_3}{k}\A_2(X)[\oo](x)
\end{align*}
for $x\in\R$. 
Hence,~$(p_2-p_1)|_{\G_f^{c_\infty}}$ and~$(p_3-p_2)|_{\G_h}$ are constant functions if and only if $\oo$ is the unique solution to \eqref{RES?}.
In this case we may choose the constants~$c_i$, $1\leq i\leq 3,$ to achieve that~$\eqref{eq:Stat}_3$ is satisfied. 
 Therewith we have proven there exists at least a solution to~\eqref{eq:Stat}. \medskip
 
\noindent{Uniqueness.} In order to establish the uniqueness of the solution,
 let $(\wt v_1,\wt v_2,\wt v_3,\wt p_1,\wt p_2,\wt p_3)$ be a further solution to \eqref{eq:Stat} with the required regularity properties and 
 set~${\wt v\coloneqq \wt v_1{\bf 1}_{\0_1}+\wt v_2{\bf 1}_{\0_2}+\wt v_3{\bf 1}_{\0_3}}.$
 The main step is to show that the function  $\oo:=(\oo_1,\oo_2)\in H^{r-1}(\R)^2$  given by 
   \begin{equation}\label{jumpv}
   \begin{aligned}
    \oo_1(x)&:=\frac{1}{2}\langle(\wt v_2- \wt v_1)(x,c_\infty+f(x))|(1,f'(x)),\\[1ex]
      \oo_2(x)&:=\frac{1}{2}\langle( \wt v_3- \wt v_2)(x,h(x))|(1,h'(x)),
   \end{aligned} 
 \end{equation}
is the unique solution to \eqref{RES?} and that  $\wt v=v$, where $v=v[\oo]$ is defined in \eqref{forvelo}.

  To start, we infer from \eqref{v12rand} and \eqref{v23rand}  that the relations \eqref{jumpv} remain valid  if we replace~$\wt v$ by~$v$.
  This   together with $\eqref{eq:Stat}_4$   implies that the global velocity field~${V:=(V^1,V^2):=\wt v-v}$ belongs to ${\rm BUC}(\R^2)$. 
 Let $\Psi\coloneqq \psi_1{\bf 1}_{\ov{\0_1}}+\psi_2{\bf 1}_{\ov{\0_2}}+\psi_3{\bf 1}_{\ov{\0_3}}$, where
$\psi_i:\ov{\0_i}\to\R$ are given  by
\begin{align*}
\psi_i(z)&\coloneqq{}\int_{h(x)}^yV^1(x,s)\, ds-\int_0^x\langle V(s,h(s))|(-h'(s),1)\rangle\, ds,\qquad i=2,\, 3,\\[1ex]
\psi_1(z)&\coloneqq{}\int_{c_\infty+f(x)}^yV^1(x,s)\, ds+\psi_2(x,c_\infty+f(x)).
\end{align*}
We clearly have  $\Psi\in {\rm C}(\R^2)$.
 Additionally, using Stokes's  theorem and the relation ${\rm div \,} V=0$ in~$\0_i$, $1\leq i\leq 3,$  we deduce that~${\nabla\psi_i=(-V^2,V^1)}$ in $\mathcal{D}'(\0_i)$. 
This implies that~${\psi_i\in {\rm UC}^1(\0_i)}$.
A simple computation which uses the continuity of $\Psi$ shows that the  (distributional) gradient~${\nabla\Psi=(-V^2,V^1)}$ belongs to~${{\rm BUC}(\R^2)}$.
Given~${\varphi\in{\rm C}^\infty_0(\R^2),}$ we obtain in virtue of the latter property that
\begin{align*}
\langle\Delta\Psi,\varphi\rangle=-\int_{\R^2}\langle\nabla\Psi|\nabla\varphi\rangle\, dz=\int_{\R^2}\langle( V^2,-V^1)|\nabla\varphi\rangle\, dz=\langle \p_yV^1-\p_x V^2,\varphi\rangle.
\end{align*}
Moreover, in virtue of  $\eqref{eq:Stat}_1$,  we have  $\p_yV^1-\p_x V^2=0$ in  $\mathcal{D}'(\0_i) $ for $1\leq i\leq 3$, and  taking advantage of~${V\in{\rm BUC}(\R^2)}$ we deduce that 
$\p_yV^1-\p_x V^2=0$ in  $\mathcal{D}'(\R^2).$
Altogether we conclude that $\Delta\Psi=0$ in $\mathcal{D}'(\R^2)$.
Consequently,~$\Psi$ is the real part of a holomorphic function~${u:\C\to\C}$.
Since $u'$ is  holomorphic too and ${u'=(\p_x\Psi,-\p_y\Psi)=-(V^2,V^1)}$ is bounded and vanishes at infinity, cf. $\eqref{eq:Stat}_5$, Liouville's theorem  yields $u'=0$, and therefore $V=0$.
Hence, $\wt v=v[\oo]$, and, as shown in the first part of the proof, $\oo$ has to solve \eqref{RES?}. 
Finally, we note that $\eqref{eq:Stat}_1$ and $\eqref{eq:Stat}_3$ imply that 
 $$P\coloneqq (p_1-\wt p_1){\bf 1}_{\ov{\0_1}}+(p_2-\wt p_2){\bf 1}_{\ov{\0_2}}+(p_3-\wt p_3){\bf 1}_{\ov{\0_3}}$$
   satisfies $\nabla P=0$ in $\R^2$, meaning that $P$ is constant in $\R^2$. 
   This completes the proof.
 \end{proof}

\section{The contour integral formulation and the proof of Theorem~\ref{MT1}}\label{Sec:3}

\subsection{The contour integral formulation}
Let $r\in(3/2,2)$ be fixed and let $\cU_r$ be the open subset of $\cO_r$ identified in Theorem~\ref{T:1}.
In view of  Theorem \ref{T:1}  we conclude  that if~${X(t):=(f(t),h(t))}$ belongs to $\cU_r$ at each time instant~${t\geq0}$,
then the velocity $v(t)$ is identified  according to the formulas \eqref{forvelo} and \eqref{RES?}.
Recalling \eqref{eq:S4}, the multiphase Muskat problem \eqref{PB}  can be now recast as the nonlinear and nonlocal autonomous evolution problem  
\begin{align}\label{NNEP}
\frac{dX(t)}{dt}=\Phi(X(t)),\quad t\geq0,\qquad X(0)=X_0,
\end{align}
where $\Phi:=(\Phi_1,\Phi_2):\cU_r\subset H^{r}(\R)^2\to H^{r-1}(\R)^2$ is  given  by
\begin{equation}\label{PHI}
\Phi(X):=\B(X)[\oo].
\end{equation}
Here $\B(X)$ is the operator introduced in \eqref{OpAB} and $\oo:=(\oo_1,\oo_2)\in H^{r-1}(\R)^2$ denotes  the unique solution to the equation 
\begin{equation}\label{INVd}
(1-A_\mu\A(X))[\oo]=\Theta X'.
\end{equation}
Below we prove that the operator $\Phi$ is smooth
\begin{equation}\label{smooth}
\Phi\in {\rm C}^\infty(\cU_r, H^{r-1}(\R)^2),
\end{equation} 
see Corollary~\ref{Cor2}.
Furthermore, we show in Theorem~\ref{T:GP} that its Fr\'echet derivative $\p\Phi(X)$ considered as an unbounded operator in $H^{r-1}(\R)^2$ with domain $H^{r}(\R)^2$,
 generates an analytic semigroup in $\kL(H^{r-1}(\R)^2)$, which in the notation introduced in \cite{Am95}  writes as
 \begin{align}\label{Gen}
 -\p\Phi(X)\in\mathcal{H}(H^{r}(\R)^2,H^{r-1}(\R)^2).
\end{align}  
The  property \eqref{Gen} holds for each  $X$ belonging to the open subset $\cV_r$ of $\cU_r$, see~\eqref {cvr} below,  identified by the Rayleigh-Taylor condition.

\subsection{Smoothness of $\Phi$}
In order to establish the smoothness of $\Phi$, see \eqref{smooth},  we first introduce some notation.
Given $n,\, m\in\N$, Lipschitz continuous  maps ${a_1,\ldots, a_{m},\, b_1, \ldots, b_n:\mathbb{R}\to\mathbb{R}}$, and~${\oo\in L_2(\R)}$, we define  the singular integral operator
\begin{equation}\label{BNM}
B_{n,m}(a_1,\ldots, a_m)[b_1,\ldots,b_n,\oo](x)\coloneqq\frac{1}{\pi}\PV\int_\mathbb{R} 
\cfrac{\prod_{i=1}^{n}\big(\delta_{[x,s]} b_i / s\big)}{\prod_{i=1}^{m}\big[1+\big(\delta_{[ x, s]}  a_i / s\big)^2\big]} \frac{\oo(x- s)}{ s}\, d s.
\end{equation}
Furthermore, for the sake of brevity we set
\begin{equation}\label{Bom}
B^0_{n,m}(f)[\oo]\coloneqq{} B_{n,m}(f,\ldots,f)[f,\ldots,f,\oo]
\end{equation}
This family of  operators has been introduced in \cite{MBV19} (up to the multiplying constant $1/\pi$) in the context of the two-phase Muskat problem, but it is also important in the study of the two-phase Stokes problem, cf.~\cite{MP2021, MP2021x}.
The next result collects some fundamental properties of these operators.

\begin{lemma}\label{L:MP0} Let  $r\in(3/2 ,2)$ and  $n,\, m\in\N$.
\begin{itemize}
\item[(i)] It holds that $\big[f\mapsto B^0_{n,m}(f)\big]\in{\rm C}^{\infty} (H^r(\R),\kL(H^{r-1}(\mathbb{R}))).$
\item[(ii)] Given $n\geq 1$, $r'\in(3/2,2)$, and $a_1,\ldots, a_m \in H^r(\mathbb{R})$, there exists a positive constant~$C$, that depends only on $n,\, m$, $r$, $r'$,  and $\max_{1\leq i\leq m}\|a_i\|_{H^r}$, such that
\begin{equation*} 
\begin{aligned} 
&\| B_{n,m}(a_1,\ldots, a_m)[b_1,\ldots, b_n,\oo] -\oo B_{n-1,m}(a_1,\ldots, a_{m})[b_2,\ldots, b_n,b_1']\|_{H^{r-1}}\\[1ex]
&\hspace{3cm}\leq C \|b_1\|_{H^{r'}}\|\oo\|_{H^{r-1}}\prod_{i=2}^{n}\|b_i\|_{H^r}
\end{aligned}
\end{equation*}
for all $b_1,\ldots, b_n\in H^r(\mathbb{R})$ and $\oo\in H^{r-1}(\mathbb{R}).$
\end{itemize}
\end{lemma}
\begin{proof}
The assertion (i) is established in  \cite[Corollary C.5]{MP2021}   and (ii) in \cite[Lemma 6]{AM21x}.
\end{proof}
The importance of the  operators $B_{n,m}$  in this context is indicated by the formulas~\eqref{notationa} and~\eqref{OpAB}  in view of
\begin{equation}\label{Forab}
\bA(u)=u'B_{0,1}^0(u)-B_{1,1}^0(u)\qquad \text{and}\qquad \bB(u)=B_{0,1}^0(u)+u'B_{1,1}^0(u).
\end{equation}
In particular, Lemma \ref{L:MP0}~(i) and the algebra property of $H^{r-1}(\mathbb{R})$ imply that
\begin{align}\label{RegAB}
 [u\mapsto \bA(u)],\, [u\mapsto \bB(u)]\in{\rm C}^{\infty} (H^r(\R),\kL(H^{r-1}(\mathbb{R}))).
\end{align}

Moreover, given  $1\leq m\in\N$ and $X_i\coloneqq{}(f_i,h_i)\in \cO_r$, $1\leq i\leq m$, we set
\begin{equation}\label{CDms}
\begin{aligned}
C_m(X_1,\ldots,X_m)[\oo](x)&\coloneqq \frac{1}{\pi}\int_\R\frac{\oo(x-s)}{\prod_{i=1}^{m}\big[s^2+ (\delta_{[ x, s]}X_i)^2\big]}\, ds,\\[1ex]
 C'_m(X_1,\ldots,X_m)[\oo](x)&\coloneqq \frac{1}{\pi}\int_\R\frac{\oo(x-s)}{\prod_{i=1}^{m}\big[s^2+ (\delta'_{[ x, s]}X_i)^2\big]}\, ds,\\[1ex]
 D_m(X_1,\ldots,X_m)[\oo](x)&\coloneqq\frac{1}{\pi}\int_\R\frac{s\oo(x-s)}{\prod_{i=1}^{m}\big[s^2+ (\delta_{[ x, s]}X_i)^2\big]}\, ds,\\[1ex]
 D'_m(X_1,\ldots,X_m)[\oo](x)&\coloneqq\frac{1}{\pi}\int_\R\frac{s\oo(x-s)}{\prod_{i=1}^{m}\big[s^2+ (\delta'_{[ x, s]}X_i)^2\big]}\, ds.
\end{aligned}
\end{equation}
where  we use again the notation introduced in \eqref{notat}. 
The operators in \eqref{CDms} have been defined (up to the multiplying constant $1/\pi$) in \cite{BM21x}.
\begin{lemma}\label{L:MP1}
Given  $r\in(3/2 ,2)$,   $ m\in\N$, and $E\in\{C, \,C',\, D,\, D'\}$ we have
 \[
 \big[X\mapsto E_{m}(X,\ldots,X)\big]\in{\rm C}^{\infty} (\cO_r,\kL(L_2(\R),H^{1}(\mathbb{R}))).
 \]
\end{lemma}
\begin{proof}
This is a consequence of \cite[Lemma 2.6]{BM21x}. 
\end{proof}

In view of  the formulas \eqref{Bom} and \eqref{CDms} we may represent the operators $\A(X)=(\A_1(X),\A_2(X))$ and~$\B(X)=(\B_1(X),\B_2(X))$ defined in \eqref{OpAB} as follows
\begin{equation}\label{ForAB}
\begin{aligned}
\A_1(X)[\oo]&=\bA(f)[\oo_1]+ f'D_{1}(X)[\oo_2]-(c_\infty+f)C_1(X)[\oo_2]+C_1(X)[h\oo_2],\\[1ex]
\A_2(X)[\oo]&=h'D_{1}'(X)[\oo_1]+(c_\infty-h)C_1'(X)[\oo_1]+C_1'(X)[f\oo_1]+\bA(h)[\oo_2],\\[1ex]
\B_1(X)[\oo]&=\bB(f)[\oo_1]+  D_{1}(X)[\oo_2]+(c_\infty+f)f'C_1(X)[\oo_2]-f'C_1(X)[h\oo_2],\\[1ex]
\B_2(X)[\oo]&=D_{1}'(X)[\oo_1]-(c_\infty-h)h'C_1'(X)[\oo_1]-h'C_1'(X)[f\oo_1]+\bB(h)[\oo_2].
\end{aligned}
\end{equation}
In then follows from  \eqref{RegAB}, \eqref{ForAB}, Lemma~\ref{L:MP1}, the algebra property of $H^{r-1}(\mathbb{R}),$ and the embedding~${H^1(\R)\hookrightarrow H^{r-1}(\R)}$
that 
\begin{align}\label{RegcAB}
 [X\mapsto \A(X)],\, [X\mapsto \B(X)]\in{\rm C}^{\infty} (\cO_r,\kL(H^{r-1}(\mathbb{R})^2)).
\end{align}

We next introduce the solution operator defined by the equation \eqref{INVd}.
\begin{lemma}\label{L:oo}
Given $r\in (3/2,2)$ and $X\in\cU_r$, let $\oo(X):=\Theta(1-A_\mu\A(X))^{-1}[X']\in H^{r-1}(\mathbb{R})^2$  denote the unique solution to \eqref{INVd}. It then holds 
\[
\oo\in {\rm C}^\infty(\cU_r, H^{r-1}(\mathbb{R})^2). 
\]
\end{lemma}
\begin{proof}
The claim follows from Theorem~\ref{Tinvert} and  \eqref{RegcAB}, by using also the  smoothness of the mapping that associates to an isomorphism its inverse.
\end{proof}

We now conclude with the following result.
\begin{cor}\label{Cor2}
Given $r\in(3/2,2)$, we have $\Phi\in {\rm C}^\infty(\cU_r, H^{r-1}(\R)^2).$
\end{cor}
\begin{proof}
The claim follows from \eqref{PHI}, \eqref{RegcAB}, and Lemma~\ref{L:oo}.
\end{proof}

\subsection{The Rayleigh-Taylor condition}
The Rayleigh-Taylor condition, see \cite{ST58}, is a sign restriction on the jump of the pressure gradients in normal direction at each interface and it reads
\begin{equation}\label{RT}
\begin{aligned}
 &\p_{\nu_1} (p_2-p_1)<0 \qquad\text{on $\Gamma_f^{c_\infty}$,}\\[1ex]
 &\p_{\nu_2} (p_3-p_2)<0 \qquad\text{on $\Gamma_h$.}
\end{aligned}
\end{equation}
Assuming that $X=(f,h)\in \cU_r$, we can express in view of $\eqref{eq:S1}_1$ and Theorem~\ref{T:1} the Rayleigh-Taylor condition as follows
\begin{equation}\label{RTRef}
\Theta_1+a_\mu^1\Phi_1(X)<0 \qquad\text{and}\qquad \Theta_2+a_\mu^2\Phi_2(X)<0.
\end{equation}
Let 
\begin{equation}\label{cvr}
\cV_r:=\{X\in\cU_r\,:\, \text{$\Theta_1+a_\mu^1\Phi_1(X)<0 $ and $\Theta_2+a_\mu^2\Phi_2(X)<0$}\}.
\end{equation}
Since $\Theta_i<0$, $i=1,\, 2$, and $\Phi(0)=\oo(0)=0$, it follows by using the smoothness property~\eqref{smooth} of~$\Phi$ that~$\cV_r$ is a nonempty open subset of $\cU_r$ 
(which coincides with $\cO_r$ when $A_\mu=0$).

\subsection{The generator property}

The next goal is to show that the evolution problem \eqref{NNEP} is of parabolic type in $\cV_r$, in the sense that $\p\Phi(X)$ generates an analytic semigroup in $\kL(H^{r-1}(\R)^2)$
for each~${X\in\cV_r,}$ as the next result states.
Here we use the notation from \cite{Am95}.

\begin{thm}\label{T:GP}
Given  $r\in(3/2,2)$ and $X\in \cV_r$, we have
 \begin{align*} 
 -\p\Phi(X)\in\mathcal{H}(H^{r}(\R)^2,H^{r-1}(\R)^2).
\end{align*}   
\end{thm}
\begin{proof}
Follows from    \cite[Theorem~I.1.6.1]{Am95}, the estimate \eqref{LOW1},  and  Theorem~\ref{T:GP12} below.
\end{proof}

In order to established the results used in the proof of Theorem~\ref{T:GP},  we fix $r\in(3/2,2)$  and~${X=(f,h)\in\cV_r}$, and we set 
\begin{equation}\label{fixedom}
\oo:=(\oo_1,\oo_2)\coloneqq\oo(X),
\end{equation}
see Lemma~\ref{L:oo}. 
The Fr\'echet derivative~$\p\Phi(X)$ can be represented as a matrix operator
\begin{align*}
\p\Phi(X)
=\begin{pmatrix}
\p_f \Phi_1(X)&\p_h \Phi_1(X)\\[1ex]
\p_f \Phi_2(X)&\p_h \Phi_2(X)
\end{pmatrix}\in\kL(H^{r}(\R)^2,H^{r-1}(\R)^2).
\end{align*}
Our strategy is to show that both diagonal entries are analytic generators, see Theorem~\ref{T:GP12} below, while the off-diagonal operator  $\p_h \Phi_1(X)$ 
is a lower order operator in the sense of the estimate~\eqref{LOW1} below. 
These properties together with the  classical result \cite[Theorem~I.1.6.1]{Am95} then lead to the claim of Theorem~\ref{T:GP}.

We first consider the off-diagonal entry $\p_h \Phi_1(X)$ and we prove that it is a lower order operator.
Let therefore~${r'\in(3/2,r)}$ be fixed.
Recalling \eqref{OpAB}, we then compute
\begin{equation}\label{eqofd}
\p_h \Phi_1(X)[v]=\p_h\B_1(X)[v][\oo]+\bB(f)[\p_h\oo_1(X)[v]] +T(X)[\p_h\oo_2(X)[v]]
\end{equation}
for $ v\in H^{r}(\R) $, where, in view of $\eqref{ForAB}_3$, we have
\begin{equation}\label{For111}
\begin{aligned}
\p_h\B_1(X)[v][\oo]&=\p_hD_1(X)[v][\oo_2]+(c_\infty+f)f'\p_hC_1(X)[v][\oo_2]\\[1ex]
&\hspace{0,45cm}-f'\p_hC_1(X)[v][h\oo_2]-f'C_1(X)[v\oo_2]
\end{aligned}
\end{equation}
for all $v\in H^{r}(\R)$.
Moreover, differentiation of the first component of   \eqref{INVd} with respect to $h$  leads, in view of \eqref{OpAB} and $\eqref{ForAB}_1$, to
\begin{equation}\label{3.18}
\begin{aligned}
(1-a_\mu^1\bA(f))[\p_h\oo_1(X)[v]]&=a_\mu^1\big(S(X)[\p_h\oo_2(X)[v]] +f'\p_hD_1(X)[v][\oo_2]+C_1(X)[v\oo_2]\\[1ex]
&\hspace{1cm}+\p_hC_1(X)[v][h\oo_2]-(c_\infty+f)\p_hC_1(X)[v][\oo_2]\big).
\end{aligned}
\end{equation}
Before proceeding with the estimates, we recall from Lemma~\ref{L:MP1} (with $r=r'$) that 
\begin{equation}\label{lowreg1}
\p_hE_1(X)\in\kL(H^{r'}(\R),\kL(L_2(\R), H^1(\R)))\qquad\text{for $E\in\{C, \,C',\, D,\, D'\}$}.
\end{equation} 
Moreover, Lemma~\ref{L:oo} (with $r=r'$)  implies  the Fr\'echet derivative $\p\oo(X)$ satisfies
\begin{equation}\label{lowreg2}
\p\oo(X)\in\kL(H^{r'}(\R)^2, H^{r'-1}(\R)^2).
\end{equation} 
Using \eqref{lowreg1} and  Lemma~\ref{L:MP1}, we infer from \eqref{For111} that
\begin{equation}\label{est1}
\begin{aligned}
\|\p_h\B_1(X)[v][\oo]\|_{H^{r-1}}&\leq C(\|\p_hD_1(X)[v][\oo_2]\|_{H^1}+\|\p_hC_1(X)[v][\oo_2]\|_{H^1}\\[1ex]
&\hspace{1cm}+\|\p_hC_1(X)[v][h\oo_2]\|_{H^1}+\|C_1(X)[v\oo_2]\|_{H^1})\\[1ex]
&\leq C\|v\|_{H^{r'}}.
\end{aligned}
\end{equation}
Additionally, $\eqref{ForAB}_3$, Lemma~\ref{L:MP1},  and \eqref{lowreg2} lead to
\begin{equation}\label{est2}
\begin{aligned}
\|T(X)[\p_h\oo_2(X)[v]]\|_{H^{r-1}}&\leq C(\| D_1(X)[\p_h\oo_2(X)[v]]\|_{H^1}+\| C_1(X)[\p_h\oo_2(X)[v]]\|_{H^1}\\[1ex]
&\hspace{1cm}+\|C_1(X)[h\p_h\oo_2(X)[v]]\|_{H^1})\\[1ex]
&\leq C\|\p_h\oo_2(X)[v]\|_{L_2}\\[1ex]
&\leq C\|\p_h\oo_2(X)[v]\|_{H^{r'-1}}\\[1ex]
&\leq C\|v\|_{H^{r'}}.
\end{aligned}
\end{equation}
Arguing similarly as above, it also holds that the right hand-side of \eqref{3.18} can be estimated by the same quantity, hence 
$$\|(1-a_\mu^1\bA(f))[\p_h\oo_1(X)[v]]\|_{H^{r-1}}\leq C\|v\|_{H^{r'}},$$
 and the isomorphism property~$\eqref{isom}_2$ (with~${u=f}$ and ${\lambda=1/a_\mu^1}$) now yields
\begin{align*}
\|\p_h\oo_1(X)[v]\|_{H^{r-1}}\leq C\|v\|_{H^{r'}}.
\end{align*}
Combining this property  with~\eqref{regAB}, we obtain
\begin{equation}\label{est3}
\|\bB(f)[\p_h\oo_1(X)[v]]\|_{H^{r-1}}\leq C\|\p_h\oo_1(X)[v]\|_{H^{r-1}}\leq C\|v\|_{H^{r'}}.
\end{equation}
 Gathering \eqref{eqofd} and \eqref{est1}-\eqref{est3}, we conclude that 
\begin{equation*}
\|\p_h \Phi_1(X)[v]\|_{H^{r-1}}\leq  C\|v\|_{H^{r'}}\qquad\text{for all $v\in H^{r}(\R)$.}
\end{equation*}
 Young's inequality implies now that for each constant $\nu>0$  there exists a constant~$C(\nu)>0$ such that 
\begin{equation}\label{LOW1}
\|\p_h \Phi_1(X)[v]\|_{H^{r-1}}\leq \nu\|v\|_{H^{r}}+C(\nu) \|v\|_{H^{r-1}}\qquad\text{for all $v\in H^{r}(\R)$.}
\end{equation}

The estimate \eqref{LOW1} is the first ingredient in the proof of Theorem~\ref{T:GP}.
We   prove  below that the diagonal entries of $\p\Phi(X)$ are both analytic generators.
As the next result shows, the generator property for $ \p_f\Phi_1(X)$ is established when merely assuming that  the Rayleigh-Taylor condition is satisfied at the interface $\Gamma_f^{c_\infty}$, respectively 
the generation property for $ \p_h\Phi_2(X)$ uses only  the Rayleigh-Taylor condition on $\Gamma_h$. 
\begin{thm}\label{T:GP12}\phantom{a}

\begin{itemize}
\item[{\rm (i)}] Assume that $\Theta_1+a_\mu^1\Phi_1(X)<0$. Then $-\p_f\Phi_1(X)\in\mathcal{H}(H^{r}(\R),H^{r-1}(\R))$.\\[-1ex]
\item[{\rm (ii)}] Assume that $\Theta_2+a_\mu^2\Phi_2(X)<0$. Then $-\p_h\Phi_2(X)\in\mathcal{H}(H^{r}(\R),H^{r-1}(\R))$.
\end{itemize}
\end{thm}

The proof of Theorem~\ref{T:GP12} is postponed to the end of the section as it requires some  preparation. 
To start, we differentiate \eqref{PHI} with respect to $f$ to arrive, in view of \eqref{OpAB}, at the formula 
\begin{equation}\label{derphi1f}
\p_f \Phi_1(X)[u]=\p_f\B_1(X)[u][\oo]+\bB(f)[\p_f\oo_1(X)[u]] +T(X)[\p_f\oo_2(X)[u]]
\end{equation}
for $ u \in H^{r}(\R)$, where, in view of $\eqref{ForAB}_3$, we have
\begin{equation}\label{derb1f}
\p_f\B_1(X)[u][\oo]=\p\bB(f)[u][\oo_1]+a_1(X)u'+ T_{\rm lot}^1[u],
\end{equation}
with
\begin{align*}
a_1(X)&\coloneqq (c_\infty+f)C_1(X)[\oo_2]-C_1(X)[h\oo_2]\in H^1(\R),\\[1ex]
T_{\rm lot}^1[u]&\coloneqq \p_fD_1(X)[u][\oo_2] +(c_\infty+f)f'\p_fC_1(X)[u][\oo_2]+ uf'C_1(X)[\oo_2]-f'\p_fC_1(X)[u][h\oo_2].
\end{align*}
The term  $T_{\rm lot}^1[u]$ is a sum of lower order terms  since \eqref{lowreg1} and  Lemma~\ref{L:MP1}  imply 
\begin{equation}\label{estf1}
\|T_{\rm lot}^1[u]\|_{H^{r-1}}\leq C\|u\|_{H^{r'}} \qquad\text{for all $u\in H^r(\R)$.}
\end{equation}
We next differentiate the first component of   \eqref{INVd} with respect to $f$ to obtain, in view of~$\eqref{ForAB}_1$, that
\begin{equation}\label{dera1f}
(1-a_\mu^1\bA(f))[\p_f\oo_1(X)[u]]=\Theta_1u'+a_\mu^1\big(\p\bA(f)[u][\oo_1]+a_2(X)u'+ T_{\rm lot}^2[u]\big),
\end{equation}
where
\begin{align*}
a_2(X)&\coloneqq  D_1(X)[\oo_2]\in H^1(\R),\\[1ex]
T_{\rm lot}^2[u]&\coloneqq S(X)[\p_f\oo_2(f)[u]]+f'\p_fD_1(X)[u][\oo_2] -(c_\infty+f)\p_fC_1(X)[u][\oo_2]\\[1ex]
&\hspace{0.5cm}- uC_1(X)[\oo_2]+\p_fC_1(X)[u][h\oo_2].
\end{align*}
Also $T_{\rm lot}^2[u]$ is a sum of lower order terms since \eqref{SSS}, \eqref{lowreg1}, \eqref{lowreg2}, and  Lemma~\ref{L:MP1} combined yield
\begin{equation}\label{estf2}
\|T_{\rm lot}^2[u]\|_{H^{r-1}}\leq C\|u\|_{H^{r'}}\qquad\text{for all $u\in H^r(\R)$.}
\end{equation}

We now consider the continuous path $\Psi:[0,1]\to \kL(H^r(\R), H^{r-1}(\R))$ defined by
\begin{equation}\label{pathpsi}
\Psi(\tau):=\tau\p_f\B_1(X)[u][\oo]+\bB(\tau f)[w_1(\tau)[u]] +\tau T(X)[\p_f\oo_2(X)[u]]
\end{equation}
where $w_1:[0,1]\to\kL(H^r(\R), H^{r-1}(\R))$ is also a  continuous path which is given by
\begin{equation}\label{pathwtf}
\begin{aligned}
(1-a_\mu^1\bA(\tau f))[w_1(\tau)[u]]&=\Theta_1 u'+ a_\mu^1\big(\tau\p\bA(f)[u][\oo_1]+\tau a_2(X)u'\\[1ex]
&\hspace{2.27cm}+ \tau T_{\rm lot}^2[u]+(1-\tau)u'\Phi_1(X)\big).
\end{aligned}
\end{equation}

With respect to the definitions \eqref{pathpsi} and \eqref{pathwtf} we include the following remarks.

\begin{rem}\label{R:2}\phantom{a}

\begin{itemize}
\item[(i)] If $\tau=1$, then $w_1(1)=\partial_f\overline{\omega}_1(X)$ and  ${\Psi(1)=\partial_f\Phi_1(X).}$\\[-2ex]
\item[(ii)] Letting $H$ denote the Hilbert transform,  we have 
\[
w_1(0)=(\Theta_1+a_\mu^1\Phi_1(X))\frac{d}{dx} \qquad\text{and}\qquad \Psi(0)=H[w_1(0)].
\] 
It is worthwhile to point out  that  the term  $(1-\tau)a_\mu^1 u'\Phi_1(X)$ in the definition~\eqref{pathwtf} is a term introduced artificially and is very important for the following facts. 
When $\tau=0$, we obtain due to this term a negative coefficient function -- which is exactly the function from the Rayleigh-Taylor condition on $\Gamma_f^{c_\infty}$ -- for the differential operator $w_1(0)$.
This aspect is important when establishing the invertibility of $\lambda-\partial_f\Phi_1(X)$ for sufficiently large  $\lambda$, see the proof of Theorem~\ref{T:GP12} below.  
Besides, below we localize the operator $\Psi(\tau)$ and show that it can locally approximated by certain Fourier multipliers, see Theorem~\ref{T:AP}.
Thanks to this artificial  term the Fourier multipliers have a coefficient which is the product of a positive function with the 
 function from the Rayleigh-Taylor condition on $\Gamma_f^{c_\infty}$ (both frozen at a certain point), see~\eqref{defAjtau}-\eqref{defalphabeta} below.
These features enable us to show that the Fourier multipliers are generators of analytic semigroups and that they satisfy certain uniform estimates, see~\eqref{L:FM1}-\eqref{L:FM2} below.
\item[(iii)] The Hilbert transform $H$ is the Fourier multiplier with symbol  $[\xi\mapsto-i\,{\rm sign}(\xi)]$ and moreover ${H\circ(d/dx)=(- d^2/dx^2)^{1/2}.}$
\item[(iv)] The properties $\eqref{isom}_2$ and \eqref{RegAB}  (both with $r=r'$) together with \eqref{estf2} and \eqref{pathwtf} imply there exists a constant $C>0$ such that 
\begin{align}\label{we1}
\|w_1(\tau)[u]\|_{H^{r'-1}}\leq C\|u\|_{H^{r'}}, \qquad u\in H^{r'}(\mathbb{R}),\, \tau\in[0,1].
\end{align}
\end{itemize}
\end{rem} 

As a further step we   locally approximate in Theorem~\ref{T:AP} below the operator  $\Psi(\tau)$,  $\tau\in[0,1]$,   by  certain Fourier multipliers~$\bA_{j,\tau}$.
To this end we associate to each given $\e\in(0,1)$, a positive integer $N=N(\e)$ and a so-called finite~$\e$-localization family
\[\{(\pi_j^\e,x_j^\e)\,:\, -N+1\leq j\leq N\}\]
 such that
\begin{align*}
\bullet\,\,\,\, \,\,&\text{$\pi_j^\e\in {\rm C}^\infty(\mathbb{R},[0,1]),$ $-N+1\leq j\leq N$, and $\sum_{j=-N+1}^N(\pi_j^\e)^2=1;$}\\[1ex]
\bullet\,\,\,\, \,\,  & \text{$ \supp \pi_j^\e $ is an interval of length $\e$ for all $|j|\leq N-1$  and $ \supp \pi_{N}^\e\subset\{|x|\geq 1/\e\}$;} \\[1ex]
\bullet\,\,\,\, \,\, &\text{ $ \pi_j^\e\cdot  \pi_l^\e=0$ if $[|j-l|\geq2, \max\{|j|, |l|\}\leq N-1]$ or $[|l|\leq N-2, j=N];$} \\[1ex]
 \bullet\,\,\,\, \,\, &\text{$\|(\pi_j^\e)^{(k)}\|_\infty\leq C\e^{-k}$ for all $ k\in\N, -N+1\leq j\leq N$;} \\[1ex]
 \bullet\,\,\,\, \,\, &x_j^\e\in\supp\pi_j^\e,\; |j|\leq N-1. 
 \end{align*} 
 The  number $x_N^\e$ plays no role in the analysis.
Corresponding to each  $\e$-localization family we define a  norm on $H^s(\mathbb{R}),$ $s\geq 0$, which is  equivalent to the standard  norm.
Indeed,   given~${s\geq0}$ and~$\e\in(0,1)$, there exists a constant $c=c(\e,s)\in(0,1)$ such that
\begin{align}\label{EQNO}
c\|u\|_{H^s}\leq \sum_{j=-N+1}^N\|\pi_j^\e u\|_{H^s}\leq c^{-1}\|u\|_{H^s},\qquad u\in H^s(\mathbb{R}).
\end{align}
 
The Fourier multipliers~$\bA_{j,\tau}$ mentioned above  are defined by
\begin{equation} \label{defAjtau}
 \begin{aligned} 
\bA_{j,\tau }&:=\bA_{j,\tau}^\e:=\alpha_\tau(x_j^\e) \Big(-\frac{d^2}{d x^2}\Big)^{1/2}+\beta_\tau (x_j^\e)\frac{d}{d x}, \quad |j|\leq N-1,\\
      \bA_{N,\tau }&:=\bA_{N,\tau }^\e:= \Theta_1\Big(-\frac{d^2}{d x^2}\Big)^{1/2},
 \end{aligned}
 \end{equation}
 where
  \begin{equation}\label{defalphabeta}
 \alpha_\tau:=\frac{1+(1-\tau)f'^2}{1+f'^2}(\Theta_1+a_\mu^1\Phi_1(X)), \qquad  \beta_\tau:=\tau B_{1,1}^0(f)[\oo_1]+\tau a_1(X)+\frac{\tau a_\mu^1\oo_1}{1+f'^2}.  
 \end{equation}

The next results provides an estimate for the localization error and is the main step in the proof of Theorem~\ref{T:GP12}~(i). 
Here we closely follow the proof of \cite[Theorem 7]{AM21x} and  we benefit from the fact that the operators $\bA(f)$ and $\bB(f)$ and their Fr\'echet derivatives have been  already localized there.

\begin{thm}\label{T:AP} 
Let $\mu>0$ be given and fix $r'\in (3/2,r)$. 
Then, there exist $\e\in(0,1)$ and  a positive constant $K=K(\e)$ such that 
 \begin{equation}\label{D1}
  \|\pi_j^\e \Psi(\tau) [u]-\bA_{j,\tau}[\pi^\e_j u]\|_{H^{r-1}}\leq \mu \|\pi_j^\e u\|_{H^r}+K\|  u\|_{H^{r'}}
 \end{equation}
 for all $-N+1\leq j\leq N$, $\tau\in[0,1],$  and  $u\in H^r(\mathbb{R})$. 
\end{thm}

Before proving Theorem~\ref{T:AP} we first present some auxiliary results used in the proof.
We start with an estimate for the commutator  $[B_{n,m}^0(f),\varphi]$.

\begin{lemma}\label{L:AL1} 
Let $n,\, m \in \N$,  $r\in(3/2, 2)$, $f\in H^r(\R)$, and  ${\varphi\in {\rm C}^1(\R)}$ with uniformly continuous derivative $\varphi'$ be given. 
Then, there exists  a constant $K$ that depends only on $ n,$ $m, $ $\|\varphi'\|_\infty, $ and~$\|f\|_{H^r}$  such that 
 \begin{equation*} 
  \|\varphi B_{n,m}^0(f)[\oo]- B_{n,m}^0(f)[ \varphi \oo]\|_{H^{1}}\leq K\| \oo\|_{2}
 \end{equation*}
for all   $\oo\in L_2(\R)$.
\end{lemma}
\begin{proof}
This result is established in \cite[Lemma 12]{AM21x}.
\end{proof}

The next lemmas describe how to localize the operator $B_{n,m}^0(f)$ (or the product of this operator with a $H^{r-1}$-function).
\begin{lemma}\label{L:AL2} 
Let $n,\, m \in \N$, $3/2<r'<r<2$, and  $\nu\in(0,\infty)$ be given. 
Let further~${f\in H^r(\mathbb{R})}$ and  $ a\in \{1\}\cup H^{r-1}(\mathbb{R})$.
For any sufficiently small $\e\in(0,1)$, there exists
a constant $K$ that depends only on $\e,\, n,\, m,\, \|f\|_{H^r},$ and $\|a\|_{H^{r-1}}$ (if $a\neq1$)  such that 
 \begin{equation*} 
  \Big\|\pi_j^\e a B_{n,m}^0(f)[ \oo]-\frac{a(x_j^\e)(f'(x_j^\e))^n}{[1+(f'(x_j^\e))^2]^m}H[\pi_j^\e \oo]\Big\|_{H^{r-1}}\leq \nu \|\pi_j^\e  \oo\|_{H^{r-1}}+K\| \oo\|_{H^{r'-1}} 
 \end{equation*}
for all $|j|\leq N-1$ and  $\oo\in H^{r-1}(\mathbb{R})$.
\end{lemma}  
\begin{proof}
See \cite[Lemma~13]{AM21x}. 
\end{proof}

Lemma~\ref{L:AL4} and Lemma \ref{L:AL5}   extend the result of Lemma \ref{L:AL2} to the case~$j=N$.

\begin{lemma}\label{L:AL4} 
Let $n,\, m \in \N$,  $3/2<r'<r<2$, and  $\nu\in(0,\infty)$ be given. 
Let further~${f\in H^r(\mathbb{R})}$ and  $a\in  H^{r-1}(\mathbb{R})$.
For any sufficiently small  $\e\in(0,1)$, there exists a constant~$K$ that depends only on~$\e,\, n,\, m,\, \|f\|_{H^r},$ and $\|a\|_{H^{r-1}}$ such that 
  \begin{equation*}
  \|\pi_N^\e a B_{n,m}^0(f)[\oo]\|_{H^{r-1}}\leq \nu \|\pi_N^\e \oo\|_{H^{r-1}}+K\| \oo\|_{H^{r'-1}}
 \end{equation*} 
 for all $\oo\in H^{r-1}(\mathbb{R})$.
\end{lemma}  
\begin{proof}
See  \cite[Lemma~14]{AM21x}. 
\end{proof}

Lemma \ref{L:AL5} is the counterpart of Lemma~\ref{L:AL4} in the case when $a=1$.

\begin{lemma}\label{L:AL5} 
Let $n,\, m \in \N$,  $3/2<r'<r<2$, and  $\nu\in(0,\infty)$ be given. 
Let further~${f\in H^r(\mathbb{R})}$.
For any sufficiently small  $\e\in(0,1)$, 
there exists a constant~$K$ that depends only on~${\e,\, n,\, m,}$ and~$\|f\|_{H^r}$  such that 
 \begin{equation*} 
  \|\pi_N^\e B_{0,m}^0(f)[ \oo]-H[\pi_N^\e \oo]\|_{H^{r-1}}\leq \nu \|\pi_N^\e \oo\|_{H^{r-1}}+K\| \oo\|_{H^{r'-1}}
 \end{equation*}
 and
  \begin{equation*}
  \|\pi_N^\e B_{n,m}^0(f)[\oo]\|_{H^{r-1}}\leq \nu \|\pi_N^\e \oo\|_{H^{r-1}}+K\| \oo\|_{H^{r'-1}},\qquad n\geq 1,
 \end{equation*}
 for  all $\oo\in H^{r-1}(\mathbb{R})$.
\end{lemma}
\begin{proof}
See  \cite[Lemma~15]{AM21x}.
\end{proof}

\pagebreak

We are now in a position to prove Theorem~\ref{T:AP}.

\begin{proof}[Proof of Theorem~\ref{T:AP}]
Let $\e\in(0,1)$ be given. Let further $\{(\pi_j^\e,x_j^\e)\,:\, -N+1\leq j\leq N\}$  be a finite $\e$-localization family and 
  $\{\chi_j^\e\,:\, -N+1\leq j\leq N\}$  a  second family with the following properties:
\begin{align*}
\bullet\,\,\,\, \,\,  &\text{$\chi_j^\e\in{\rm C}^\infty(\mathbb{R},[0,1])$ and $\chi_j^\e=1$ on $\supp \pi_j^\e$, $-N+1\leq j\leq N$;} \\[1ex]
\bullet\,\,\,\, \,\,  &\text{$\supp \chi_j^\e$ is an interval  of length $3\e$, $|j|\leq N-1$, and $\supp\chi_N^\e\subset \{|x|\geq 1/\e-\e\}$.}  
\end{align*} 
In the arguments that follow we use several times  the  estimate
\begin{align}\label{MES}
\|uv\|_{H^{r-1}}\leq  C(\|u\|_\infty\|v\|_{H^{r-1}}+\|v\|_\infty\|u\|_{H^{r-1}})
\end{align} 
which holds for all $u,\, v\in H^{r-1}(\mathbb{R})$ and $r\in(3/2,2)$,  with  $C$ independent of $u$ and $v$.
Moreover, we denote by~$C$ constants that do not depend on~$\e$, while constants~$K$  may  depend on~$\e$.
We divided the proof in four main steps.
\medskip

\noindent{Step 1: The term $\p_f\B_1(X)[u][\oo]$.}  In view of \cite[Equations (4.19) and (4.20)]{AM21x}, if $\varepsilon$ is sufficiently small, then we have
\begin{equation}\label{T1aj}
\begin{aligned}
  &\Big\|\pi_j^\varepsilon\partial\mathbb{B}(f)[u][\overline{\omega}_1] 
  + \frac{\overline{\omega}_1(x_j^\varepsilon) f'(x_j^\varepsilon)}{1+f'^2(x_j^\varepsilon)}H[(\pi_j^\varepsilon u)']
  -B_{1,1}^0(f)[\overline{\omega}_1](x_j^\e)(\pi_j^\varepsilon u)'\Big\|_{H^{r-1} } \\[1ex]
  &\hspace{2cm}\leq \frac{\mu}{4}\|\pi_j^\varepsilon u\|_{H^{r}} +K\|u\|_{H^{r'}}
\end{aligned}
\end{equation}
for all $ |j|\leq N-1$ and  $u\in H^{r}(\mathbb{R}) ,$ and 
\begin{align}\label{T1aN}
  \|\pi_N^\varepsilon\partial\mathbb{B}(f)[u][\overline{\omega}_1] \|_{H^{r-1}} \leq \frac{\mu}{4}\|\pi_N^\varepsilon u\|_{H^{r}} +K\|u\|_{H^{r'}}
\end{align}
for   all $u\in H^r(\mathbb{R})$.

Moreover, the estimate \eqref{MES} together with the identity $\chi_j^\e\pi_j^\e=\pi_j^\e$, $-N+1\leq j\leq N$, yields, in view of  $a_1(X)\in {\rm C}^{1/2}(\R),$   that 
\begin{equation}\label{T1bj}
\begin{aligned}
\|\pi_j^\e a_1(X)u'- a_1(X)(x_j^\e)(\pi^\e_j u)'\|_{H^{r-1}}&\leq\|\chi_j^\e \big(a_1(X)-a_1(X)(x_j^\e)\big)(\pi^\e_j u)'\|_{H^{r-1}}+K\|u\|_{H^{r-1}}\\[1ex]
&\leq C\|\chi_j^\e \big(a_1(X)-a_1(X)(x_j^\e)\big)\|_\infty\|\pi^\e_j u\|_{H^{r}}+K\|u\|_{H^{r'}}\\[1ex]
&\leq  \frac{\mu}{4} \|\pi_j^\e u\|_{H^r}+K\|  u\|_{H^{r'}},\qquad |j|\leq N-1,
\end{aligned}
\end{equation}
if $\e$ is sufficiently small, respectively, by taking into account that $a(X)$ vanishes at infinity, 
\begin{equation}\label{T1bN}
\begin{aligned}
\|\pi_N^\e a_1(X)u' \|_{H^{r-1}}&\leq\|\chi_N^\e a_1(X)(\pi_N^\e u)'\|_{H^{r-1}}+K\|u\|_{H^{r-1}}\\[1ex]
&\leq C\|\chi_N^\e a_1(X)\|_{\infty}\|\pi^\e_N u\|_{H^{r}}+K\|u\|_{H^{r'}}\\[1ex]
&\leq  \frac{\mu}{4} \|\pi_N^\e u\|_{H^r}+K\|  u\|_{H^{r'}}
\end{aligned}
\end{equation}
for all  $u\in H^r(\R).$

Recalling \eqref{estf1}, we have
\begin{equation}\label{T1c}
\| \pi_j^\e T_{\rm lot}^1[u]\|_{H^{r-1}}\leq K\|u\|_{H^{r'}}
\end{equation}
for all $u\in H^r(\R) $  and $-N+1\leq j\leq N$, and therewith we have localized all three summands  in the formula~\eqref{derb1f} for~$\p_f\B_1(X)[u][\oo]$.\medskip

\noindent{Step 2: The term $T(X)[\p_f\oo_2(X)[u]]$.} Combining \eqref{SSS} and \eqref{lowreg2}, we have
\begin{equation}\label{T2}
\begin{aligned}
\| \pi_j^\e T(X)[\p_f\oo_2(X)[u]]\|_{H^{r-1}}&\leq K\|T(X)[\p_f\oo_2(X)[u]]\|_{H^{r-1}}\leq K\| \p_f\oo_2(X)[u] \|_{2}\\[1ex]
&\leq K\| \p_f\oo_2(X)[u] \|_{H^{r'-1}}\leq K\| u \|_{H^{r'}}
\end{aligned}
\end{equation}
for all $u\in H^r(\R) $, $\e\in(0,1)$,  and $-N+1\leq j\leq N$.\medskip

\noindent{Step 3: The term $\bB(\tau f)[w_1(\tau)[u]]$.}  We divide  this   step into two substeps.\medskip

\noindent{Step 3a.} We prove there exists a positive  constant $C_0$  such that  
\begin{align}\label{HT3}
\|\pi_j^\varepsilon w_1(\tau)[u]\|_{H^{r-1}}\leq  C_0\|\pi_j^\varepsilon u\|_{H^{r}}+ K\|u\|_{H^{r'}}
\end{align}
for all    $\e\in(0,1)$, $-N+1\leq j\leq N$, $\tau\in[0,1]$, and $u\in H^r(\R) $.
To this end we multiply~\eqref{pathwtf} by~$\pi_j^\e$  and arrive at
\begin{equation}\label{HT3a}
\begin{aligned}
(1-a_\mu^1\bA(\tau f))[\pi_j^\e w_1(\tau)[u]]&=\Theta_1 \pi_j^\e u'+ a_\mu^1\big(\pi_j^\e\bA(\tau f)[ w_1(\tau)[u]]-\bA(\tau f)[\pi_j^\e w_1(\tau)[u]]\big)\\[1ex]
&\hspace{0,45cm}+a_\mu^1\pi_j^\e\big(\tau\p\bA(f)[u][\oo_1]+\tau a_2(X)u'\\[1ex]
&\hspace{2.27cm}+ \tau T_{\rm lot}^2[u]+(1-\tau)u'\Phi_1(X)\big).
\end{aligned}
\end{equation}
In view of $\eqref{isom}_2$ and \eqref{RegAB}, it remains to show that  the $H^{r-1}$-norm of the right side of \eqref{HT3a} may be estimated by the right side of \eqref{HT3}.
To start, we infer from Lemma~\ref{L:AL1} and \eqref{we1} that
\begin{equation}\label{HT3b}
\|\pi_j^\e\bA(\tau f)[ w_1(\tau)[u]]-\bA(\tau f)[\pi_j^\e w_1(\tau)[u]]\|_{H^{r-1}}\leq  K\|w_1(\tau)[u]\|_{2} \leq  K\|u\|_{H^{r'}}
\end{equation}
for all $\e\in(0,1)$, $-N+1\leq j\leq N$, $\tau\in[0,1]$, and $u\in H^r(\R)$.
Furthermore, taking into account that~${a_2(X)\in H^1(\R)}$ and~$\Phi_1(X)\in H^{r-1}(\R)$, we further have
\begin{equation}\label{HT3c}
\|\Theta_1 \pi_j^\e u'+a_\mu^1\pi_j^\e(\tau a_2(X)u'+(1-\tau)u'\Phi_1(X))\|_{H^{r-1}}\leq C\|\pi_j^\e u'\|_{H^{r-1}} \leq  C\|\pi_j^\e u\|_{H^{r}}+K\|u\|_{H^{r'}}.
\end{equation}
Recalling \eqref{estf2}, we get
\begin{equation}\label{HD3d}
\| \pi_j^\e T_{\rm lot}^2[u]\|_{H^{r-1}}\leq K\|u\|_{H^{r'}},
\end{equation}
and it remains  to estimate the term $\pi_j^\e \tau\p\bA(f)[u][\oo_1],$ where
\begin{equation*}
\begin{aligned}
  \pi\partial\mathbb{A}(f)[u][\overline{\omega}_0]&=u'B_{0,1}(f)[\overline{\omega}_1]-2f'B_{2,2}(f,f)[f,u,\overline{\omega}_1]\\[1ex]
  &\hspace{0.424cm}-B_{1,1}(f)[u,\overline{\omega}_1]+2B_{3,2}(f,f)[f,f,u,\overline{\omega}_1], \quad u\in H^r(\mathbb{R}),
\end{aligned}
\end{equation*}
cf. \cite[Equation 4.7]{AM21x}.
Invoking Lemma~\ref{L:MP0}~(ii), we deduce that
\begin{equation}\label{derrepA}
 \partial\mathbb{A}(f)[u][\overline{\omega}_1]=u'B_{0,1}(f)[\overline{\omega}_1]+\overline{\omega}_1\big(-2f'B_{1,2}^0(f)[u']-B_{0,1}^0(f)[u']+2B_{2,2}^0(f)[u']\big)+T_{\rm lot}^3[u],
\end{equation}
where 
\begin{equation}\label{derrepA'}
\| T_{\rm lot}^3[u]\|_{H^{r-1}}\leq C\|u\|_{H^{r'}}.
\end{equation}
Since for $n,\, m\in\N$ we have
\begin{align*}
 \|\pi_j^\e B_{n,m}^0(f)[u']\|_{H^{r-1}}&\leq  \| B_{n,m}^0(f)[\pi_j^\e u']\|_{H^{r-1}}+\|[B_{n,m}^0(f),\pi_j^\e][u']\|_{H^{r-1}}\\[1ex]
 &\leq C\|\pi_j^\e u'\|_{H^{r-1}}+K\|u'\|_2\\[1ex]
 &\leq C\|\pi_j^\e u\|_{H^{r}}+K\|u\|_{H^{r'}},
\end{align*}
cf. Lemma~\ref{L:MP0}~(i) and Lemma~\ref{L:AL1}, we conclude that 
\begin{equation}\label{HT3e}
\|\pi_j^\e \tau\p\bA(f)[u][\oo_1]\|_{H^{r-1}} \leq  C\|\pi_j^\e u\|_{H^{r}}+K\|u\|_{H^{r'}}.
\end{equation}
Gathering \eqref{HT3a}-\eqref{HT3e}, it now follows from   $\eqref{isom}_2$ and \eqref{RegAB} that \eqref{HT3} indeed holds true.\medskip

\noindent{Step 3b.} Let $C_0$ be the constant from \eqref{HT3}. In view of \eqref{Forab}, \eqref{we1}, and    Lemma~\ref{L:AL2} (when~${|j|\leq N-1}$), respectively   
 Lemma~\ref{L:AL4} and Lemma \ref{L:AL5} (when $j=N$),  for sufficiently small~$\e$   we have
\begin{equation}\label{T3a}
\begin{aligned}
\|\pi_j^\varepsilon \bB(\tau f)[w_1(\tau)[u]]-H[\pi_j^\e w_1(\tau)[u]]\|_{H^{r-1}}&\leq  \frac{\mu}{4C_0}\|\pi_j^\varepsilon w_1(\tau)[u]\|_{H^{r-1}}+ K\|w_1(\tau)[u]\|_{H^{r'-1}}\\[1ex]
&\leq  \frac{\mu}{4}\|\pi_j^\varepsilon u\|_{H^{r}}+ K\|u\|_{H^{r'}}
\end{aligned}
\end{equation}
for all  $-N+1\leq j\leq N$, $\tau\in[0,1]$, and $u\in H^r(\R).$
We further define 
\[
\varphi_\tau:=\Theta_1+\tau a_\mu^1 a_2(X)+(1-\tau)a_\mu^1\Phi_1(X)+\tau a_\mu^1B_{0,1}(f)[\oo_1], \qquad \tau\in[0,1].
\]
We prove below that  if $\e$ is sufficiently small, then
\begin{equation}\label{T3aj}
\Big\|H[\pi_j^\e w_1(\tau)[u]]-\varphi_\tau(x_j^\e)H[(\pi_j^\e u)']-\frac{\tau a_\mu^1\oo_1(x_j^\e)}{1+f'^2(x_j^\e)}(\pi_j^\e u)'\Big\|_{H^{r-1}}\leq   \frac{\mu}{4}\|\pi_j^\varepsilon u\|_{H^{r}}+ K\|u\|_{H^{r'}}
\end{equation}
for all $ |j|\leq N-1$, $\tau\in[0,1]$, and $u\in H^r(\R)$, respectively that 
\begin{equation}\label{T3bN}
\|H[\pi_N^\e w_1(\tau)[u]]-\Theta_1 H[(\pi_N^\e u)']\|_{H^{r-1}}\leq   \frac{\mu}{4}\|\pi_N^\varepsilon u\|_{H^{r}}+ K\|u\|_{H^{r'}}
\end{equation}
for all $\tau\in[0,1]$ and $u\in H^r(\R).$
Indeed, in view of   $H^2=-{\rm id}_{H^{r-1}(\R)}$  and ${\|H\|_{\kL(H^{r-1}(\R))}=1}$, we get
\begin{equation*}
\begin{aligned}
&\hspace{-0.5cm}\Big\|H[\pi_j^\e w_1(\tau)[u]]-\varphi_\tau(x_j^\e)H[(\pi_j^\e u)']-\frac{\tau a_\mu^1\oo_1(x_j^\e)}{1+f'^2(x_j^\e)}(\pi_j^\e u)'\Big\|_{H^{r-1}}\\[1ex]
&\leq  \Big\| \pi_j^\e w_1(\tau)[u] -\varphi_\tau(x_j^\e) (\pi_j^\e u)' +\frac{\tau a_\mu^1\oo_1(x_j^\e)}{1+f'^2(x_j^\e)}H[(\pi_j^\e u)']\Big\|_{H^{r-1}}
\end{aligned}
\end{equation*}
for $|j|\leq N-1$, respectively
\begin{equation*}
\|H[\pi_N^\e w_1(\tau)[u]]-\Theta_1 H[(\pi_N^\e u)']\|_{H^{r-1}}\leq   \|\pi_N^\e w_1(\tau)[u]-\Theta_1 (\pi_N^\e u)'\|_{H^{r-1}}.
\end{equation*}
In order to estimate the right sides of the latter two estimates let first~${|j|\leq N-1}$. Multiplying   the equation \eqref{pathwtf} by $\pi_j^\e,$ we arrive at
\begin{equation*}
\pi_j^\e w_1(\tau)[u] -\varphi_\tau(x_j^\e) (\pi_j^\e u)' +\frac{\tau a_\mu^1\oo_1(x_j^\e)}{1+f'^2(x_j^\e)}H[(\pi_j^\e u)']=T_1+T_2+T_3+T_4,
\end{equation*}
where
\begin{align*}
T_1&\coloneqq \pi_j^\e(\Theta_1+\tau a_\mu^1 a_2(X)+(1-\tau)a_\mu^1\Phi_1(X)) u'-(\Theta_1+\tau a_\mu^1 a_2(X)+(1-\tau)a_\mu^1\Phi_1(X))(x_j^\e) (\pi_j^\e u)',\\[1ex]
T_2&\coloneqq  \tau a_\mu^1\Big( \pi_j^\e\p\bA(f)[u][\oo_1]-B_{0,1}^0(f)[\oo_1](x_j^\e)(\pi_j^\e u)'+\frac{\oo_1(x_j^\e)}{1+f'^2(x_j^\e)}H[(\pi_j^\e u)']\Big),\\[1ex]
T_3&\coloneqq a_\mu^1\pi_j^\e \bA(\tau f)[w_1(\tau)[u]],\\[1ex]
T_4&\coloneqq \tau a_\mu^1\pi_j^\e T_{\rm lot}^2[u].
\end{align*}
Since $a_2(X),\, \Phi_1(X)\in {\rm C}^{r-3/2}(\R),$ the arguments used to derive \eqref{T1bj}   together with \eqref{estf2} yield
\[
\|T_1\|_{H^{r-1}}+\|T_4\|_{H^{r-1}}\leq  \frac{\mu}{12}\|\pi_j^\varepsilon u\|_{H^{r}}+ K\|u\|_{H^{r'}}.
\]
Furthermore, recalling \eqref{derrepA} and \eqref{derrepA'},  the arguments used  to derive \eqref{T1bj} and repeated use of Lemma~\ref{L:AL2} lead  to
\[
\|T_2\|_{H^{r-1}}\leq  \frac{\mu}{12}\|\pi_j^\varepsilon u\|_{H^{r}}+ K\|u\|_{H^{r'}}.
\] 
Finally, combining \eqref{Forab}, Lemma~\ref{L:AL2},   \eqref{we1}, and \eqref{HT3} we get
\begin{align*}
\|T_3\|_{H^{r-1}}&\leq |a_\mu^1|\Big\| \tau f'B_{0,1}^0(\tau f)[w_1(\tau)[u]]- \frac{\tau f'(x_j^\e)}{1+\tau ^2f'^2(x_j^\e)}H[\pi_j^\e w_1(\tau)[u]]\Big\|_{H^{r-1}}\\[1ex]
&\hspace{0,45cm}+|a_\mu^1|\Big\|  B_{1,1}^0(\tau f)[w_1(\tau)[u]]- \frac{\tau f'(x_j^\e)}{1+\tau ^2f'^2(x_j^\e)}H[\pi_j^\e w_1(\tau)[u]]\Big\|_{H^{r-1}}\\[1ex]
&\leq  \frac{\mu}{12}\|\pi_j^\varepsilon u\|_{H^{r}}+ K\|u\|_{H^{r'}},
\end{align*}
 provided that $\e$ is sufficiently small.  Herewith we established \eqref{T3aj}. 
 
 Let now $j= N$. Multiplying   the equation \eqref{pathwtf} by $\pi_N^\e$ we arrive at
\begin{equation*}
\pi_N^\e w_1(\tau)[u] -\Theta_1 (\pi_N^\e u)' =T_5+T_6,
\end{equation*}
where
\begin{align*}
T_5&\coloneqq \pi_N^\e(\Theta_1+\tau a_\mu^1 a_2(X)+(1-\tau)a_\mu^1\Phi_1(X)) u'-\Theta_1 (\pi_N^\e u)' ,\\[1ex]
T_6&\coloneqq  \tau a_\mu^1\pi_N^\e\p\bA(f)[u][\oo_1]+ a_\mu^1\pi_N^\e \bA(\tau f)[w_1(\tau)[u]]+ \tau a_\mu^1\pi_N^\e T_{\rm lot}^2[u].
\end{align*}
Since both $a_2(X)$ and $ \Phi_1(X)$ vanish at infinity, we obtain, by arguing as in the derivation of~\eqref{T1bN}, that    
\[
\|T_5\|_{H^{r-1}} \leq  \frac{\mu}{8}\|\pi_N^\varepsilon u\|_{H^{r}}+ K\|u\|_{H^{r'}}.
\]
Furthermore, the relations \eqref{Forab}, \eqref{estf2}, \eqref{HT3}, \eqref{derrepA},  \eqref{derrepA'},     the fact that~${B_{0,1}(f)[\oo_1]\in H^{r-1}(\R)}$ vanishes at infinity, Lemma~\ref{L:AL4}, and Lemma~\ref{L:AL5} lead us to
\[
\|T_6\|_{H^{r-1}}\leq  \frac{\mu}{8}\|\pi_N^\varepsilon u\|_{H^{r}}+ K\|u\|_{H^{r'}},
\]
 provided that $\e$ is sufficiently small.  This proves  \eqref{T3bN}.
 
Combining \eqref{T3a} and  \eqref{T3aj}, we conclude that if $\e$ is sufficiently small, then
  \begin{equation}\label{T3j}
\Big\|\pi_j^\varepsilon \bB(\tau f)[w_1(\tau)[u]]
-\varphi_\tau(x_j^\e)H[(\pi_j^\e u)']-\frac{\tau a_\mu^1\oo_1(x_j^\e)}{1+f'^2(x_j^\e)}(\pi_j^\e u)'\Big\|_{H^{r-1}}\leq   \frac{\mu}{2}\|\pi_j^\varepsilon u\|_{H^{r}}+ K\|u\|_{H^{r'}}
\end{equation}
for all $|j|\leq N-1$, $\tau\in[0,1],$ and $u\in H^r(\R) $.
If $j=N$, we conclude from \eqref{T3a} and \eqref{T3bN} that, if $\e$ is sufficiently small, then
 \begin{equation}\label{T3N}
\|\pi_N^\varepsilon \bB(\tau f)[w_1(\tau)[u]]-\Theta_1 H[(\pi_N^\e u)']\|_{H^{r-1}}\leq   \frac{\mu}{2}\|\pi_N^\varepsilon u\|_{H^{r}}+ K\|u\|_{H^{r'}}
\end{equation}
for all  $\tau\in[0,1]$ and $u\in H^r(\R) $.
\medskip

\noindent{Step 4.}  The desired claim \eqref{D1} follows, in the case $j=N$, directly from \eqref{T1aN}, \eqref{T1bN}, \eqref{T1c}, \eqref{T2}, and \eqref{T3N}. 
If $|j|\leq N-1$, we infer from \eqref{T1aj}, \eqref{T1bj}, \eqref{T1c}, \eqref{T2}, and \eqref{T3j} that the claim \eqref{D1} holds true, but for the coefficient function $\alpha_\tau$ defined  in \eqref{defalphabeta} we obtain the formula
\[
\alpha_\tau\coloneqq \Theta_1-\frac{\tau \oo_1 f'}{1+f'^2} +\tau a_\mu^1a_2(X)+(1-\tau)a_\mu^1\Phi_1(X)+\tau a_\mu^1B_{0,1}^0(f)[\oo_1],\qquad \tau\in[0,1].
\]
However, recalling the definition \eqref{fixedom} of $\oo$, the definitions of $a_i(X),$ $i=1,\, 2$, the definition~\eqref{PHI} of $\Phi_1(X),$  and the formulas \eqref{Forab} and \eqref{ForAB}, we derive the following relations
\begin{align*}
\oo_1&=\Theta_1f'+a_\mu^1\big(f'B_{0,1}(f)[\oo_1] -B_{1,1}^0(f)[\oo_1]+f'a_2(X)-a_1(X)\big),\\[1ex]
\Phi_1(X)&=B_{0,1}^0(f)[\oo_1]+f'B_{1,1}^0(f)[\oo_1]+a_2(X)+f'a_1(X).
\end{align*}
Replacing $\oo_1$ in the second term of the formula for $\alpha_\tau$ by the expression found above, we get, by using the identity for $\Phi_1(X)$, that $\alpha_\tau$ can be  indeed  expressed as in  \eqref{defalphabeta}.
This completes the proof.
\end{proof}

We now consider the Fourier multipliers defined in \eqref{defAjtau} more closely. The Rayleigh-Taylor condition $\Theta_1+a_\mu^1\Phi_1(X)<0$, together with the fact that 
$f'$, $\Phi_1(X)$, $\oo_1$, $a_1(X)$, $B_{1,1}^0(f)[\oo_1]$ all belong to $ H^{r-1}(\R)$
implies there exists $\eta\in(0,1)$ such that the coefficient functions $\alpha_\tau$ and $\beta_\tau$ of the Fourier multipliers, cf. \eqref{defalphabeta},  satisfy
\begin{align*}
\eta\leq -\alpha_\tau\leq \frac{1}{\eta}\quad\text{and}\quad \|\beta_\tau||_\infty\leq \frac{1}{\eta}\qquad\text{for all $\tau\in[0,1]$.}
\end{align*}
Consequently,  classical Fourier analysis arguments imply there exists $\kappa_0=\kappa_0(\eta)\geq 1$ such that 
\begin{align}
\bullet &\quad \mbox{$\lambda-\bA_{\alpha,\beta}\in \kL(H^r(\R),H^{r-1}(\R))$ is an isomoprphism for all $ \re\lambda\geq 1,$}\label{L:FM1}\\[1ex]
\bullet &\quad  \kappa_0\|(\lambda-\bA_{\alpha,\beta})[u]\|_{H^{r-1}}\geq |\lambda|\cdot\|u\|_{H^{r-1}}+\|u\|_{H^r}, \qquad \forall\, u\in H^r(\R),\, \re\lambda\geq 1\label{L:FM2},
\end{align}
uniformly for $\bA_{\alpha,\beta}\coloneqq{} \alpha(-d^2/dx^2)^{1/2}+\beta (d/dx)$ with $-\alpha\in[\eta,1/\eta],$ $|\beta|\leq 1/\eta$.
The properties~\eqref{L:FM1}-\eqref{L:FM2} together with our previous results enable us to obtain the desired generator property for~${\p_f\Phi_1(X)}$.

\begin{proof}[Proof of Theorem~\ref{T:GP12}]
(i) Arguing as in \cite[Theorem~4.1]{AM21x}, we  find  in view of \eqref{EQNO}, \eqref{L:FM1}-\eqref{L:FM2}, and of Theorem~\ref{T:AP} constants~$\kappa=\kappa(X)\geq1$  and $\omega=\omega(X)>0 $ such that 
  \begin{align}\label{KDED}
   \kappa\|(\lambda-\Psi(\tau))[u]\|_{H^{r-1}}\geq |\lambda|\cdot\|u\|_{H^{r-1}}+ \|u\|_{H^{r}}
 \end{align}
for all   $\tau\in[0,1],$   $\re \lambda\geq \omega$, and  $u\in H^{r}(\R)$.
Furthermore, recalling Remark \ref{R:2}~(ii), we infer from \cite[Proposition 1]{AM21x} that~${\omega-\Psi(0)\in {\rm Isom\,}(H^{r}(\R), H^{r-1}(\R))}$.
The method of continuity, see e.g.~\cite[Proposition I.1.1.1]{Am95}, and \eqref{KDED} imply that~${\omega-\Psi(1)=\omega-\p_f\Phi_1(X)}$ also belongs to~${\rm Isom\,}(H^{r}(\R), H^{r-1}(\R))$. 
In view of this property and of \eqref{KDED} (with $\tau=1$) we finally conclude that~$-\p_f\Phi_1(X)\in\mathcal{H}(H^{r}(\R),H^{r-1}(\R))$, cf. \cite[Chapter~I]{Am95}.  

(ii) The generator property for $\p_h\Phi_2(X)$ follows by using similar arguments and therefore we omit its proof (see \cite{JBThesis} for details). 
\end{proof}

 \subsection{The proof of the main result}

We finally come to the proof of our main result which, in view of the abstract parabolic theory presented in \cite[Chapter 8]{L95}, is now obtained as a
consequence of the smoothness property established in Corollary~\ref{Cor2} and of the fact that the  evolution problem~\eqref{NNEP}
is parabolic in $\cV_r,$ cf. Theorem~\ref{T:GP}.
\begin{proof}[Proof of Theorem~\ref{MT1}]
 Given $\alpha\in(0,1)$,   $T>0$, and a Banach space $X$  we set
 \begin{align*}
 {\rm C}^{\alpha}_{\alpha}((0,T], X):=\Big\{f:(0,T]\longrightarrow X\,:\,\|f\|_{C_\alpha^\alpha}:=\sup_t\|f(t)\|+\sup_{s\neq t}\frac{\|t^\alpha f(t)-s^\alpha f(s)\|}{|t-s|^\alpha}<\infty\Big\}.
 \end{align*}
Corollary~\ref{Cor2} and Theorem \ref{T:GP} ensure  that the assumptions of \cite[Theorem~8.1.1]{L95} 
are all satisfied in the context of  the evolution problem~\eqref{NNEP}.
Applying this theorem,  we find for each~${X\in \cV_r}$, a local solution $X(\cdot;X_0)$ to \eqref{NNEP} such that
\[ X\in {\rm C}([0,T],\cV_r)\cap {\rm C}^1([0,T], H^{r-1}(\mathbb{R})^2)\cap {\rm C}^{\alpha}_{\alpha}((0,T], H^r(\mathbb{R})^2),\] 
where   $T=T(X_0)>0$ and   $\alpha\in(0,1)$ is fixed (but arbitrary).
This solution is unique within the set
\[
  \bigcup_{\beta\in(0,1)}{\rm C}^{\beta}_{\beta}((0,T],H^r(\mathbb{R})^2) \cap {\rm C}([0,T],\cV_r)\cap {\rm C}^1([0,T], H^{r-1}(\mathbb{R})^2).
 \]
As stated in Theorem~\ref{MT1}, the uniqueness   holds true in  
${\rm C}([0,T],\cV_r)\cap {\rm C}^1([0,T], H^{r-1}(\mathbb{R})^2).$
Indeed, let $X\in {\rm C}([0,T],\cV_r)\cap {\rm C}^1([0,T], H^{r-1}(\mathbb{R})^2)$ be a solution to \eqref{NNEP}, let~${r'\in (3/2,r)}$   be fixed, and   set~${\alpha:= r-r'\in(0,1)}$. 
It then holds 
 \begin{equation*}
\|X(t_1)-X(t_2)\|_{H^{r'}}\leq \|X(t_1)-X(t_2)\|_{H^{r-1}}^{\alpha}\|X(t_1)-X(t_2)\|_{H^{r}}^{1-\alpha}   \leq C|t_1-t_2|^\alpha,\quad t_1,\, t_2\in[0, T],
 \end{equation*}
which shows in particular that  $X \in   {\rm C}^{\alpha}_{\alpha}((0,T], H^{r'}(\mathbb{R})^2)$.
Applying the uniqueness statement of \cite[Theorem 8.1.1]{L95}  in the context of~\eqref{NNEP} with 
$\Phi\in {\rm C}^{\infty}(\cV_{r'}, H^{r'-1}(\mathbb{R})^2),$
now implies the desired uniqueness assertion.
This unique solution can be extended up to a maximal existence time~${T^+(X_0)}$, see \cite[Section 8.2]{L95}.
Moreover, it follows from the proof of Theorem~\ref{T:1} that the velocities and the pressures enjoy all the properties mentioned in Theorem~\ref{MT1}.  \medskip
 
 (i) The continuous dependence of the solution on the initial data  follows from~\cite[Proposition~8.2.3]{L95}. \medskip
  
  (ii) The parabolic smoothing properties follow by using a parameter trick which was successfully applied also to other problems, see  \cite{An90, ES96, PSS15, MBV19, AM21x}.
Since the details are very similar to those in \cite[Theorem~2~(ii)]{AM21x}  we omit them.
\end{proof}

\subsection*{Acknowledgement}
The authors gratefully acknowledge the support by the RTG 2339
 ``Interfaces, Complex Structures, and Singular Limits'' of the German Science Foundation (DFG).

 \bibliographystyle{siam}
\bibliography{JB2}
\end{document}